\documentclass[10pt]{article}
\usepackage{amsmath}
\usepackage{amsfonts}
\usepackage{amssymb}
\usepackage{graphicx}
\usepackage{mathrsfs}
\usepackage{xcolor}
\usepackage{bbm}
\usepackage{verbatim}
\usepackage{dsfont}
\usepackage{mathrsfs}
\usepackage[body={15.5cm,21cm}, top=3cm]{geometry}
\usepackage{paralist}
\usepackage{indentfirst}
\allowdisplaybreaks[4]
\usepackage{hyperref}
\hypersetup{hypertex = true,
colorlinks=true,
linkcolor=blue,
anchorcolor=blue,
citecolor =blue }
\providecommand{\U}[1]{\protect \rule{.1in}{.1in}}
\numberwithin{equation}{section}
\newtheorem{Theorem}{Theorem}[section]

\newtheorem{corollary}[Theorem]{Corollary}

\newtheorem{definition}[Theorem]{Definition}

\newtheorem{lemma}[Theorem]{Lemma}

\newtheorem{remark}[Theorem]{Remark}

\newenvironment{proof}[1][Proof]{\noindent \textbf{#1.} }{\  \rule{0.5em}{0.5em}}
\allowdisplaybreaks[2]
\begin{document}
\title{A unified approach to global solvability for FBSDEs with diagonal generators}
\author{Tianjiao Hua \thanks{School of Mathematical Sciences, Shanghai Jiao Tong University, China (htj960127@sjtu.edu.cn).}
\and
Peng Luo \thanks{School of Mathematical Sciences, Shanghai Jiao Tong University, China (peng.luo@sjtu.edu.cn). Research supported by the
National Natural Science Foundation of China (No. 12101400).}}
\maketitle
\begin{abstract}
In this paper, we study the global solvability of  multidimensional forward–backward
stochastic differential equations (FBSDEs) with diagonally Lipschitz, quadratic or super-quadratic generators. Under a certain “monotonicity” condition, we provide a unified approach which shows that there exists a decoupling field that is uniformly Lipschitz in its spatial variable. This decoupling field is closely related to bounded solution to an associated characteristic BSDE. For Lipschitz case, we provide some extensions and investigate $L^p$-solution and $L^p$ estimates. Our results gives a positive answer to a question proposed in Yong (Banach Center Publ. 122: 255-286, 2020). Applications to stochastic optimal controls and stochastic differential games are investigated.
\end{abstract}

\textbf{MSC-classification}: 60H30, 93H20.
\section{Introduction}
A fully coupled FBSDE takes the following form:
\begin{equation}
 \left\{\begin{aligned}
X_{t} &=  x+\int_{0}^{t} b\left(s, X_{s}, Y_{s}, Z_{s}\right) d s+\int_{0}^{t} \sigma\left(s, X_{s}, Y_{s}, Z_{s}\right) d W_{s}, \\
Y_{t} &=  g\left(X_{T}\right)+\int_{t}^{T} f\left(s, X_{s}, Y_{s}, Z_{s}\right) d s-\int_{t}^{T} Z_{s} d W_{s}, \quad t\in [0,T]
\end{aligned}\right. \label{FBSDE3} 
\end{equation}
for a given initial value $x$ and a multidimensional Brownian motion $W$. This system naturally appears in numerous areas of applied mathematics including stochastic control and mathematical finance.\\
\indent The theory of (fully) coupled FBSDEs started in early 1990's. Using the \textit{Method of Contraction Mapping}, Antonelli \cite{antonelli1993} obtained the first result on the solvability of
FBSDEs \eqref{FBSDE3} when the duration T is relatively small and later  detailed in  \cite{pardoux1999forward}. However, arbitrary time duration case is more involved. There have been three main methods to treat FBSDEs with arbitrary duration: the \textit{Four Step Scheme} by Ma–Protter–Yong \cite{ma1994}, the \textit{Method of Continuation}
 by Hu-Peng \cite{hu1995solution} and Peng-Wu \cite{peng1999}, and the \textit{decoupling field method} by Ma-Wu-Zhang-Zhang \cite{ma2015}. \textit{Four step scheme} requires the non-degeneracy of the forward diffusion and the non-randomness of the
coefficients since it makes use of  quasilinear partial differential equation; \textit{Method of Continuation} requires essentially the “monotonicity” condition on the coefficients, which is restrictive in a different way. \\
\indent The \textit{decoupling field method} 
 \cite{ma2015} aims to find a function $u$ as in the \textit{Four step scheme}, such that
\begin{equation}
    Y_{t} = u(t,X_{t})\label{decoupling}.
\end{equation}
However $u$ can be a random field, then the method can solve general non-Markovian FBSDEs without the non-degeneracy of the forward diffusion. The key issue is the existence of a decoupling field that is uniformly Lipschitz in its spatial variable. Zhang \cite{zhang2006} obtained global solution of FBSDE \eqref{FBSDE3} in the case $\sigma = \sigma (t,x,y)$. The idea was later extended by Ma-Wu-Zhang-Zhang \cite{ma2015}, which gives a set of sufficient conditions to get the uniform bound by studying two dominating ODEs in one-dimensional case. Multidimensional case is further studied by Fromm and Imkeller in \cite{fromm2013existence} and Zhang in \cite{zhang2017wellposedness}.\\
\indent In this paper, we give a unified approach to solve FBSDEs with diagonal generators for which we can show the existence of a decoupling field that
is uniformly Lipschitz in its spatial variable under some monotonicity conditions. Based on this result, we obtain global solutions for FBSDEs with diagonal Lipschitz, diagonal quadratic and diagonal super-quadratic generators. Our approach is motivated by the works of Ma-Wu-Zhang-Zhang \cite{ma2015}. By borrowing the method of decoupling field, we show that the global solvability of FBSDEs \eqref{FBSDE3} with $b=b(t,x,y)$ and $ \sigma = \sigma(t)$ is closely related to the uniform boundedness of the value process of an associated BSDE. Under monotonicity conditions and diagonal structure of the generators, we obtain the uniform bound of the value process of the associated BSDE by using comparison theorem for multidimensional BSDEs. For Lipschitz generators, global solvability results are obtained for more general cases. \\
\indent Our results contribute to the literature in the following ways. First, we provide a new kind of monotonicity conditions to solve  FBSDEs globally. In Lipschitz case, we extend part of the results of \cite{ma2015} to multidimensional case.  We are able to solve some FBSDEs where the monotonicity conditions required in \cite{peng1999} are not satisfied (see Remark \ref{continuation remark}). Compared with \cite{antonelli2006existence}, we can get the uniqueness of the solution. Moreover, we relax the conditions needed in \cite{zhang2017wellposedness}. Second, our method does not need the non-degeneracy of the diffusion process, which is often necessary to get global solution with PDE method in Markovian setting (see \cite{delarue2002existence,kupper2019}). Therefore, our results may be applied to a wider range of problems. Specially, global solvability for a special time-delayed BSDE is presented. Third, we obtain global solutions for quadratic and super-quadratic FBSDEs. To the best of our knowledge, few works exist considering the existence of solutions of quadratic or super-quadratic FBSDEs, particularly, the existence of global solutions. We extend the results of \cite{jackson2021global,kupper2019,luo2017solvability} in different aspects (see Remark \ref{quadratic remark}). Finally, for Lipschitz case, we establish $L^{p}$-solution and $L^{p}$ estimates globally, which give a positive answer to a question proposed in \cite{yong2020p}.  Some applications to stochastic optimal controls and stochastic differential games are investigated.\\
\indent The rest of the paper is organized as follows. In section 2, we present the notations. In section 3, we first revisit some well-posedness results for FBSDEs with diagonal Lipschitz, diagonal quadratic and diagonal super-quadratic generators over small time duration, and then prove the existence and uniqueness over arbitrary large time interval. Section 4 gives some extensions for Lipschitz case: on one hand, we deal with a larger set of  FBSDEs, and on the other one, we study $L^{p}$-solution and establish $L^{p}$ estimates. In section 5, we discuss the connection between FBSDEs and BSDEs with time-delayed generators. Section 6 is denoted to give applications to stochastic optimal controls and stochastic differential games.
\section{Preliminaries}
\indent We work on a filtered probability space $\left(\Omega, \mathcal{F},\left(\mathcal{F}_{t}\right)_{t \in[0, T]}, P\right)$ with $T \in(0, \infty)$. We assume that the filtration is generated by a $d$-dimensional Brownian motion $W$, is complete and right continuous. Let us also assume that $\mathcal{F}=\mathcal{F}_{T}$. Unless otherwise stated, all equalities and inequalities between random variables and processes will be understood in the $P$-a.s. and $P \otimes d t$-a.e. sense, respectively. $|\cdot|$ denotes the Euclidean norm and $\langle \cdot,\cdot \rangle$ denotes the inner product. We use the exponent $\mathsf{T}$ to denote the transpose of a matrix. For $x, y \in \mathbb{R}^{n}, x \leq y$ is understood component-wisely, i.e., $x \leq y$ if and only if $x^{i} \leq y^{i}$ for all $i=1, \ldots, n$. For $p \ge 1$, we denote by
\begin{itemize}
    \item $\mathcal{S}^{\infty}\left(\mathbb{R}^{n}\right)$ the set of $n$-dimensional continuous adapted processes $Y$ on $[0, T]$ such that
\begin{equation*}
\|Y\|_{\mathcal{S}}{ }_{\left(\mathbb{R}^{n}\right)}:=\left\|\sup _{0 \leq t \leq T}\left|Y_{t}\right|\right\|_{\infty}<\infty ;
\end{equation*}
\item $L^{p}\left(\mathcal{F}_{t} ; \mathbb{R}^{n}\right)$ the set of $n$-dimensional $\mathcal{F}_{t}$-measurable random variables $\xi$ such that
\begin{equation*}
E\left[ |\xi|^{p}\right]^{\frac{1}{p}}<\infty;
\end{equation*}
\item $L^{\infty}\left(\mathcal{F}_{t} ; \mathbb{R}^{n}\right)$ the set of $n$-dimensional $\mathcal{F}_{t}$-measurable random variables $\xi$ such that
\begin{equation*}
\|\xi\|_{\infty}<\infty;
\end{equation*}
\item  $\mathcal{S}^{p}\left(\mathbb{R}^{n}\right)$ the set of adapted and continuous processes $X$ valued in $\mathbb{R}^{n}$ such that
\begin{equation*}
    \|X\|_{\mathcal{S}^{p}\left(\mathbb{R}^{n}\right)}^{p}:=E\left[\sup _{0 \leq t \leq T} \left|X_{t}\right|^{p}\right]<\infty;
\end{equation*}
\item $\mathcal{H}^{p}\left(\mathbb{R}^{n \times d}\right)$ the set of predictable processes $Z$ valued in $\mathbb{R}^{n \times d}$ such that
\begin{equation*}
    \|Z\|_{\mathcal{H}^{p}\left(\mathbb{R}^{n \times d}\right)}^{p}:=E\left[\left(\int_{0}^{T}\left|Z_{u}\right|^{2} d u\right)^{p / 2}\right]<\infty.
\end{equation*}
\end{itemize}
\indent For a suitable integrand $Z$, we denote by $Z \cdot W$ the stochastic integral $\left(\int_{0}^{t} Z_{u} d W_{u}\right)_{t \in[0, T]}$ of $Z$ with respect to $W$. From Protter \cite{protter2005stochastic}, $Z \cdot W$ defines a continuous martingale for any $Z \in \mathcal{H}^{p}\left(\mathbb{R}^{n \times d}\right)$. Let us further define by $\mathrm{BMO}_{p}$, with $p \in[1, \infty)$, the space of martingales $M$ valued in $\mathbb{R}^{n}$ such that
\begin{equation*}
\|M\|_{\mathrm{BMO}_{p}}:=\sup _{\tau}\left\|E\left[\left|M_{T}-M_{\tau}\right|^{p} d u \mid \mathcal{F}_{\tau}\right]^{1 / p}\right\|_{\infty}<\infty,
\end{equation*}
where the supremum is taken over all stopping times valued in $[0, T]$. In the sequel, we will denote $\mathrm{BMO}$ the space $\mathrm{BMO}_{2}$. The properties of the BMO space and norm can be found in \cite{kazamaki2006continuous}. Throughout the paper, for any $x \in \mathbb{R}$ and any function $\phi(x)$, we will use the following convention
 \begin{equation*}
    \frac{\phi(x)-\phi(x)}{x-x}:=0.
\end{equation*}  
\section{FBSDEs with diagonal generators}
\indent In this section, we consider the following coupled forward-backward stochastic differential equations (FBSDEs)
\begin{equation}
  \left\{\begin{array}{l}
X_t=x+\int_0^t b\left(s, X_s, Y_s\right) d s+\int_0^t \sigma_s d W_s, \\
Y_t^i=h^i\left(X_T\right)+\int_t^T f^i\left(s, X_s, Y_s, Z_s^i\right) d s-\int_t^T Z_s^i d W_s, \quad i=1,2, \cdots, n,
\end{array}\right. \label{FBSDE} 
\end{equation}
where the generator $f$ has a diagonal structure and is allowed to be Lipschitz, quadratic and super-quadratic in $Z$.\\
\indent Let $K$ be a positive constant, we will make the following assumptions.
\begin{itemize}
    \item [\textbf{(H)}]
    \begin{itemize}
        \item [(i)] $b: \Omega \times[0, T] \times \mathbb{R} \times \mathbb{R}^{n} \rightarrow \mathbb{R}$ is progressive measurable and
\begin{equation*}
\left|b(t,x, y)-b\left(t,x^{\prime}, y^{\prime}\right)\right| \leq K(\left|x-x^{\prime}\right|+\left|y-y^{\prime}\right|)
\end{equation*}
  for all $x, x^{\prime} \in \mathbb{R}$ and $y, y^{\prime} \in \mathbb{R}^{n}$.
  \item [(ii)] $h: \Omega \times \mathbb{R} \rightarrow \mathbb{R}^{n}$ is $\mathcal{F}_{T}$-measurable and
\begin{equation*}
\left|h(x)-h\left(x^{\prime}\right)\right| \leq K\left|x-x^{\prime}\right|
\end{equation*}
for all $x, x^{\prime} \in \mathbb{R}$.
\item[(iii)] $f:\Omega \times [0, T] \times \mathbb{R} \times \mathbb{R}^{n} \times \mathbb{R}^{n \times d} \rightarrow \mathbb{R}^{n}$ is progressive measurable and $f^{i}(t,x,y,z) = f^{i}(t,x,y,z^{i})$ for $i=1,2,\cdots,n$.
    \end{itemize}
    \item[\textbf{(A1)}]
    \begin{itemize}
        \item[(i)] The function $f(t,\cdot, \cdot, \cdot)$ is continuous for each $t \in[0, T]$ such that
\begin{equation*}
\left|f(t,x, y, z)-f\left(t,x^{\prime}, y^{\prime}, z^{\prime}\right)\right| \leq K(\left|x-x^{\prime}\right|+|y-y^{\prime}|+|z-z^{\prime}|)
\end{equation*}
for all $x, x^{\prime} \in \mathbb{R}, y,y^{\prime} \in \mathbb{R}^{n}$ and $z,z^{\prime} \in \mathbb{R}^{n \times d}$.
\item[(ii)]The following integrability condition holds
\begin{equation*}
    E\left[\left(\int_{0}^{T}|b(t,0,0)|dt\right)^{2}+\left(\int_{0}^{T}|f(t,0,0,0)|dt\right)^{2}+\int_{0}^{T}|\sigma(t)|^{2}dt+|h(0)|^{2}\right]<\infty.
\end{equation*}
    \end{itemize}
  \item [\textbf{(A2)}]
   \begin{itemize}
       \item[ (i) ]
       There exists $\lambda \ge 0 $ such that
        \begin{equation*}
    |h(x)|\le \lambda
         \end{equation*}
for all $x \in \mathbb{R}$.
      \item [ (ii) ] 
      $\sigma: \Omega \times[0, T] \rightarrow \mathbb{R}^{d} $ is predictable such that $\sigma \in \mathcal{H}^{2}\left(\mathbb{R}^{d}\right)$. 
      \item [ (iii) ]It holds that
\begin{equation*}
\begin{aligned}
&\left|b(t,x, y)\right| \leq K(1+|x|+|y|),\\
&\left|f(t,x,y,z)-f\left(t,x^{\prime}, y^{\prime},z^{\prime} \right)\right| \leq K(\left|x-x^{\prime}\right|+\left|y-y^{\prime}\right|)+ K\left(1+|z|+\left|z^{\prime}\right|\right)\left|z-z^{\prime}\right|,\\
&\left|f(t,x,y,z)\right| \leq K \left(1+|y|+|z|^{2}\right)
\end{aligned}
\end{equation*}
for all $x, x^{\prime} \in \mathbb{R}, y, y^{\prime} \in \mathbb{R}^{n}$ and $z, z^{\prime} \in \mathbb{R}^{n\times d}$.
  \end{itemize}
  \item [\textbf{(A3)}]
  \begin{itemize}
    \item[(i)] The functions
  $b,h,f$ are deterministic and $b$ satisfies 
  \begin{equation*}
      \left|b(t,x, y)\right| \leq K(1+|x|+|y|).
  \end{equation*}
   \item[(ii)]
       $\sigma:[0, T] \rightarrow \mathbb{R}^{d}$ is measurable and  $\left|\sigma_{t}\right| \leq K $ for all $t \in[0, T].$
       \item[(iii)]
      The function $f(t,\cdot, \cdot ,\cdot )$ is continuous for each $t \in[0, T]$, $\int_{0}^{T}\left|f(t,0,0,0)\right|^{2} dt < \infty$ and there exists a non-decreasing function $\rho: \mathbb{R}_{+} \rightarrow \mathbb{R}_{+}$ such that
\begin{equation*}
\left|f(t,x, y, z)-f\left(t,x^{\prime}, y, z\right)\right| \leq K\left|x-x^{\prime}\right|
\end{equation*}
for all $x, x^{\prime} \in \mathbb{R}, y \in \mathbb{R}^{n}$ and $z \in \mathbb{R}^{n \times d}$ such that $|z| \leq M:=8 K^{2} \sqrt{d n}$ and
\begin{equation*}
\left|f(t,x, y, z)-f\left(t,x, y^{\prime}, z^{\prime}\right)\right| \leq K\left|y-y^{\prime}\right|+\rho\left(|z| \vee\left|z^{\prime}\right|\right)\left|z-z^{\prime}\right|
\end{equation*}
for all $x \in \mathbb{R}, y, y^{\prime} \in \mathbb{R}^{n}$ and $z, z^{\prime} \in \mathbb{R}^{n \times d}$.
\item[(iv)] It holds that
\begin{equation}
\begin{aligned}
 &\left|f(t,x, y, z)-f\left(t,x^{\prime}, y, z\right)-f\left(t,x, y^{\prime}, z^{\prime}\right)+f\left(t,x^{\prime}, y^{\prime}, z^{\prime}\right)\right| \\
&\quad \leq K\left|x-x^{\prime}\right|\left(\left|y-y^{\prime}\right|+\left|z-z^{\prime}\right|\right)
\end{aligned}
\end{equation}
for all $t\in[0,T],x,x^{\prime}\in\mathbb{R},y,y^{\prime}\in \mathbb{R}^{n}$ and $z,z^{\prime} \in\mathbb{R}^{n\times d}$.
\end{itemize}
\end{itemize}
\begin{remark}
As stated in \cite{kupper2019}, the condition (A3)(iv) is the minimal condition needed to ensure Lipschitz continuity in $y, z$ of the Malliavin derivative of $f\left(t,X_t, y, z\right)$ for a given SDE solution $X$, see e.g. \cite{el1997backward,cheridito2014bsdes} for details. When the generator $f$ is of the form $f(t,x, y, z):=f^1(t,x)+f^2(t,y)+f^3(t,z)$ for some functions $f^1, f^2$ and $f^3$, then (A3)(iv) is satisfied.
\end{remark}
\indent  To guarantee the global solvability of FBSDE \eqref{FBSDE}, we further impose the following monotonicity conditions.
\begin{itemize}
\item[\textbf{(M1)}] For $1\leq i,j \leq n$, $t\in[0,T] $ and $(x,y,z), (x,\bar{y},z)\in \mathbb{R}\times \mathbb{R}^{n} \times \mathbb{R}^{n\times d}$ such that $y^{i}=\bar{y}^{i}$ and $ y^j\leq \bar{y}^j$ for $j\neq i$, then we have
 \begin{equation*}
    f^{i}(t,x,y,z^{i})\le f^{i}(t,x,\bar{y},z^{i}).
\end{equation*}
\item[\textbf{(M2)}]  For $t\in[0,T]$, one of the following two conditions is satisfied:\\
    (i)
   For $(x,y),(x,\bar{y})\in \mathbb{R}\times \mathbb{R}^{n}$, such that $y\leq \bar{y}$, we have \begin{equation*}
    b(t,x,y) \geq b(t,x,\bar{y}).
\end{equation*}
And for $1 \leq i \leq n$, $(x,y,z), (\bar{x},y,z)\in \mathbb{R}\times \mathbb{R}^{n} \times \mathbb{R}^{n\times d}$, such that $x\leq \bar{x}$,  we have
\begin{equation}
  f^{i}(t,x,y,z^{i}) \leq f^{i}(t,\bar{x},y,z^{i}), \quad h^{i}(x)\leq h^{i}(\bar{x}). 
\end{equation}\\
(ii) For $(x,y),(x,\bar{y})\in \mathbb{R}\times \mathbb{R}^{n}$, such that $y\leq \bar{y}$, we have \begin{equation*}
    b(t,x,y) \leq b(t,x,\bar{y}).
\end{equation*}
And for $1 \leq i \leq n$, $(x,y,z), (\bar{x},y,z)\in \mathbb{R}\times \mathbb{R}^{n} \times \mathbb{R}^{n\times d}$, such that $x\leq \bar{x}$,  we have
\begin{equation}
  f^{i}(t,x,y,z^{i}) \geq f^{i}(t,\bar{x},y,z^{i}), \quad h^{i}(x)\geq h^{i}(\bar{x}). 
\end{equation}
\end{itemize}
\begin{remark}
The assumption (M1) states that  $f$ is quasi-monotonicity function, which often appears in multidimensional comparison theorem of BSDEs (see \cite{wu2009comparison}). 
\end{remark}
\indent Now we introduce some notations used in this section. For ease of notations, for $y_{1},y_{2}\in \mathbb{R}^{n}$, we denote
\begin{equation*}
   (y_{1}^{(1,i)},y_{2}^{(i+1,n)}):=(y_{1}^{1},y_{1}^{2},\cdots,y_{1}^{i},y_{2}^{i+1},\cdots,y_{2}^{n}). 
\end{equation*}
For $\left(x_{1}, y_{1}, z_{1}\right),\left(x_{2}, y_{2}, z_{2}\right)\in\mathbb{R}\times\mathbb{R}^n\times\mathbb{R}^{n\times d}$ and for $i,j=1,2,\cdots,n$, $k=1,2,\cdots,d$, let $\theta_1:=\left(x_{1}, y_{1}, z_{1}\right),\theta_2:=\left(x_{2}, y_{2}, z_{2}\right)$ and denote
\begin{equation}
    \begin{aligned}
    h_{1}^{i}(x_{1},x_{2}) &\triangleq \frac{h^{i}(x_{1})-h^{i}(x_{2})}{x_{1}-x_{2}},\\
    b_{1}(t,\theta_{1},\theta_{2})&\triangleq \frac{b(t,x_{1},y_{1})-b(t,x_{2},y_{1})}{x_{1}-x_{2}},\\ 
    f_{1}^{i}(t,\theta_{1},\theta_{2}) &\triangleq \frac{f^{i}(t,x_{1},y)-f^{i}(t,x_{2},y)}{x_{1}-x_{2}},\\
    b_{2}^{j}(t,\theta_{1},\theta_{2}) &\triangleq \frac{b(t,x_{2},z_{1},y_{2}^{(1,j-1)},y_{1}^{(j,n)})-b(t,x_{2},z_{1},y_{2}^{(1,j)},y_{1}^{(j+1,n)})}{y_{1}^{j}-y_{2}^{j}},\\
   f_{2}^{ij}(t,\theta_{1},\theta_{2})&\triangleq
    \frac{f^{i}(t,x_{2},z_{1},y_{2}^{(1,j-1)},y_{1}^{(j,n)})-f^{i}(t,x_{2},z_{1},y_{2}^{(1,j)},y_{1}^{(j+1,n)})}{y_{1}^{j}-y_{2}^{j}},\\
    f_{3}^{ik}(t,\theta_{1},\theta_{2})&\triangleq
    \frac{f^{i}(t,x_{2},y_{2},z_{2}^{i(1,k-1)},z_{1}^{i(k,d)})-f^{i}(t,x_{2},y_{2},z_{2}^{i(1,k)},z_{1}^{i(k+1,d)})}{z_{1}^{ij}-z_{2}^{ij}},
    \end{aligned} \label{quadratic notation}
\end{equation}
and $b_{2}(t,\theta_{1},\theta_{2}) = (b_{2}^{1},b_{2}^{2},\cdots,b_{2}^{n})(t,\theta_{1},\theta_{2})$, $f_{2}^{i}(t,\theta_{1},\theta_{2}) = (f_{2}^{i1},f_{2}^{i2},\cdots,f_{2}^{in})(t,\theta_{1},\theta_{2})$, $f_{3}^{i}(t,\theta_{1},\theta_{2}) = (f_{3}^{i1},f_{3}^{i2},\cdots,f_{3}^{id})(t,\theta_{1},\theta_{2})$. 

With the above introduced notations, it is obvious that the conditions (M1) indicates that for $1\leq i,j \leq n$, $f_{2}^{ij}\ge 0$, $ j \neq i$ and the condition (M2) indicates the non-positivity or non-negativity of $b_{2},f^{i}_{1},h_{1}^{i}$.\\
\indent Before moving to our main results, we recall the following elementary result for ordinary differential equations, whose proof is given for completeness.
\begin{lemma}
Consider the following ordinary differential equation
\begin{equation}
    y_{t} = H + \int _{t}^{T} (Ay_{s}+Ky_{s}+B)ds, \label{ODE}
\end{equation}
where $H, B\in \mathbb{R}^n, A \in \mathbb{R}^{n\times n}$ and all the components of $H, B$ and $A$ are $K$. Then ODE \eqref{ODE} has a unique solution on $[0,T]$ satisfying $|y_t|\leq nK(T+1)e^{(n+1)KT}$. \label{lemma 3.5}
\end{lemma}
\begin{proof}
It follows from standard ODE theory that ODE \eqref{ODE} admits a unique solution satisfying $y_t\geq 0$ for all $t\in[0,T]$. 
Denote
\begin{equation*}
    \tilde{y}_{t} = \sum _{i=1}^{n}y_{t}^{i},
\end{equation*}
then we get
\begin{equation*}
   \tilde{y} _{t}=nK+\int_{t}^{T}\left[(nK+K)\tilde{y}_{s}+nK\right]ds,
\end{equation*}
and using Gronwall's inequality we obtain
\begin{equation*}
    \tilde{y}_{t} \leq nK(T+1)e^{(n+1)KT},
\end{equation*}
which indicates $|y_t|\leq nK(T+1)e^{(n+1)KT}$.
\end{proof}

\indent Our first result concerns the local existence and uniqueness of solution. Moreover, we can prove that the function $u$ defined by \eqref{ut definition} is uniformly Lipschitz continuous in its spatial variable.
\begin{Theorem}
Under assumptions (H), (M1), (M2) and if one of assumptions (A1)-(A3) holds, there exists a constant $\delta>0$, such that for any $t\in[T-\delta,T]$ and $x\in\mathbb{R}$, the following FBSDE
\begin{equation}
\left\{\begin{array}{l}
X^{t,x}_{s}=x+\int _{t}^{s} b\left(r, X^{t,x}_{r}, Y^{t,x}_{r}\right) d r+\int _{t}^{s} \sigma_{r} d W_{r}, \\
Y_{s}^{t,x,i}=h^{i}\left(X^{t,x}_{T}\right)+\int_{s}^{T} f^{i}\left(r, X^{t,x}_{r}, Y^{t,x}_{r}, Z_{r}^{t,x,i}\right) d r-\int_{s}^{T} Z_{r}^{t,x,i}d W_{r}, \quad i=1,2,\cdots,n,
\end{array}\right.\label{small FBSDE}
\end{equation}
has a unique solution on $[t,T]$. Further the random function $u$ defined on $\Omega\times[T-\delta,T]\times\mathbb{R}$ by 
\begin{equation}
    u(\omega,t,x) = Y_{t}^{t,x}(\omega),\label{ut definition}
\end{equation}
satisfies for any $t\in[T-\delta,T]$ and $s\in[t,T]$,
    \begin{equation}\label{decoupling field}
    Y_{s}^{t,x} = u(s,X^{t,x}_{s}), ~~P\text{-}a.s.
    \end{equation}
    and
    \begin{equation}
\left|\frac{u(t,x_{1})-u(t,x_{2})}{x_{1}-x_{2}}\right|\leq y_{t}\label{ut uniform}
    \end{equation}
   for all $x_1,x_2\in\mathbb{R}$ and $x_1\neq x_2$, where $y_{t}$ is the solution of ODE \eqref{ODE}.  \label{Theorem :small duration}
\end{Theorem}
\begin{proof}
   Under assumption (H), if assumption (A1) holds, it follows from \cite{ma1999book} that there exists $\delta_{1} $ only depending on $K$, such that for any $t\in[T-\delta,T]$ and $x\in\mathbb{R}$, FBSDE \eqref{small FBSDE} admits a unique solution $(X^{t,x},Y^{t,x},Z^{t,x}) \in \mathcal{S}^{2}(\mathbb{R}) \times \mathcal{S}^{2}(\mathbb{R}^{n}) \times \mathcal{H}^{2}\left(\mathbb{R}^{n \times d}\right)$; if assumption (A2) holds, it follows from  \cite{luo2017solvability} that there exists  $\delta_{2}$ only depending on $K$ and $\lambda$, such that whenever for any $t\in[T-\delta,T]$ and $x\in\mathbb{R}$, FBSDE \eqref{small FBSDE} admits a unique solution $(X^{t,x},Y^{t,x},Z^{t,x}) \in \mathcal{S}^{2}(\mathbb{R}) \times \mathcal{S}^{\infty}(\mathbb{R}^{n}) \times \mathcal{H}^{2}\left(\mathbb{R}^{n \times d}\right)$  with $\|Z^{t,x} \cdot W\|_{B M O} \leq C $, for some constant $C$; if assumption (A3) holds, it follows from  \cite{kupper2019} that there exists  $\delta_{3} $ only depending on $n,K,d$, such that $t\in[T-\delta,T]$ and $x\in\mathbb{R}$, FBSDE \eqref{small FBSDE} admits a unique solution $(X^{t,x},Y^{t,x},Z^{t,x}) \in \mathcal{S}^2\left(\mathbb{R}\right) \times \mathcal{S}^2\left(\mathbb{R}^n\right) \times \mathcal{S}^{\infty}\left(\mathbb{R}^{n \times d}\right) $ with $|Z^{t,x}|\leq M$, for some constant $M$, which indicates the super-quadratic case can be included in Lipschitz case and analyzed together in the following proof process.
   
     Next we let $\delta = \operatorname{min} \{\delta_{1},\delta_{2},\delta_{3}\}$. For any $t\in[T-\delta,T]$ and any $ x_1,x_2 \in \mathbb{R}$ satisfying $x_{1} \neq x_{2}$, we denote $\Theta^{i}=(X^{t,x_i},Y^{t,x_i},Z^{t,x_i})$ for $i=1,2$ and $s\in[t,T]$
\begin{equation}
\nabla X_{s} \triangleq \frac{X^{t,x_1}_{s}-X^{t,x_2}_{s}}{x_{1}-x_{2}},\quad \nabla Y_{s} \triangleq \frac{Y^{t,x_1}_{s}-Y^{t,x_2}_{s}}{x_{1}-x_{2}}, \quad \nabla Z_{s} \triangleq \frac{Z^{t,x_1}_{s}-Z^{t,x_2}_{s}}{x_{1}-x_{2}}.\label{theta}
\end{equation}

One can check easily that $(\nabla X, \nabla Y, \nabla Z)$ satisfies the following  “variational FBSDE” on $[t,T]$
\begin{equation}\label{variation FBSDE}
\left\{\begin{array}{rlr}
\nabla X_{s}= & 1+\int_{t}^{s}\left(b_{1}(r) \nabla X_{r}+b_{2}(r)\nabla Y_{r}\right) dr, \\
\nabla Y_{s}^{i}= & h_{1}^{i}\nabla X_{T}+\int_{s}^{T}\left(f_{1}^{i}(r) \nabla X_{r}+f_{2}^{i}(r)\nabla Y_{r}+f_{3}^{i}(r)(\nabla Z_{r}^{i})^{\mathsf{T}}\right) dr-\int_{s}^{T} \nabla Z_{r}^{i} d W_{r},\quad i=1,2,\cdots,n,
\end{array}\right.
\end{equation}
where $h_{1}^{i} \triangleq h_{1}^{i}(X_{T}^{t,x_1},X_{T}^{t,x_2})$, $ b_{j}(r)\triangleq b_{j}(r,\Theta^{1}_{r},\Theta^{2}_{r})$, $j=1,2$, $f_{j}^{i}(r)\triangleq f_{j}^{i}(r,\Theta^{1}_{r},\Theta^{2}_{r})$, $j=1,2,3$, respectively. From now on, we might suppress time variable in case no confusion occurs. We note here that $b_{1},b_{2},f_{1},f_{2},f_{3}$ are adapted processes and $h_{1}$ is an $\mathcal{F}_{T}$-measurable random variable. Moreover, by choosing a smaller $\delta$ if necessary\footnote{The dependence of $\delta$ on the constants remains the same as the above. }, we have 
\begin{equation}
   |\nabla Y_{s}|\leq C|\nabla X_{s}|, \forall s\in [t,T], P\text{-}a.s.\label{tilde Y bound}
\end{equation}
for some positive constant. Under assumption (H) and (A1) or (A3), \eqref{tilde Y bound} is directly from standard arguments, see, for example, \cite[Theorem I.5.1]{ma1999book}. And under assumptions (H) and (A2), since  $f_{3}^{i}\cdot W$ is a BMO martingale, we can use Girsanov transformation and similar argument with Lipschitz case and get \eqref{tilde Y bound}.

  Now for any $t\in[T-\delta,T]$ and $x\in\mathbb{R}$, we could define a random field $u(\omega,t, x) \triangleq Y_{t}^{t, x}(\omega)$.
 In particular, from \eqref{tilde Y bound} and following similar argument as in \cite[Corollary 1.5]{delarue2002existence}, it can be shown that
    \begin{equation}
        Y_{s}^{t, x}=u\left(s, X_{s}^{t, x}\right), \text{ for all }s \in[t, T], P\text{-}a.s.\label{U}
    \end{equation}

Denote
\begin{equation}
     \nabla u_{t} \triangleq \frac{u(t,x_{1})-u(t,x_{2})}{x_{1}-x_{2}}.\label{nable u}
\end{equation}
In particular, we have 
\begin{equation}
 \nabla Y_{t} = \nabla u_{t}\nabla X_{t}. \label{ut}
\end{equation}

We now show that $\nabla X_{s}$ remains positive on the whole interval $[t,T]$. To this end, let $\tau \in [t,T]$ be a stopping time such that $\nabla X_{s}$ is positive on $[t,\tau)$. 
 Denote
\begin{equation*}
\tilde{Y}_{s}=\nabla Y_{s}\left[\nabla X_{s}\right]^{-1};\quad s \in[t, \tau).
\end{equation*}
Note that on $[t,\tau)$, we have 
\begin{equation*}
    d \nabla X_{s}^{-1} = -\nabla X_{s}^{-1} \left[(b_{1}+b_{2}\tilde{Y}_{s})\right]ds,
\end{equation*}
then we get
\begin{equation*}
    \nabla X_{s}^{-1} = \exp\left(\int _{t} ^{s}(-b_{1}-b_{2}\tilde{Y}_{r})dr\right).
\end{equation*}

Moreover, the uniform boundedness of $\tilde{Y}$ implied from \eqref{tilde Y bound} and the Lipschitz continuity of $b$ imply $\nabla X^{-1}$ is bounded on $\left[t, \tau\right)$ by a constant that does not depend on $\tau$. This implies that $\nabla X $ can never reach 0 and therefore, we can choose $\tau=T$.\\
Since $\nabla X_{s}$ stays positive on $[t,T]$, we get from \eqref{ut} that 
\begin{equation}
    \nabla u_{t} = \nabla Y_{t}\nabla X_{t}^{-1}=\tilde{Y}_{t}.
\end{equation}
Thus \eqref{ut uniform} is equivalent to $|\tilde{Y}_{t}|\leq y_{t}$. Indeed, we can prove $| \tilde{Y}_{s}|\leq y_{s}$ for all $s\in[t,T]$.\\
Applying Ito's formula to $\nabla Y_{s}[\nabla X_{s}]^{-1}$ on $s\in[t,T]$, it holds
\begin{equation}
   \begin{aligned}
d\tilde{Y}_{s}^{i} &=\nabla Y_{s}^{i} d \nabla X_{s}^{-1}+\nabla X_{s}^{-1} d \nabla Y_{s}^{i} +d\left\langle \nabla X_{s}^{-1}, \nabla Y_{s}^{i}\right\rangle\\
&=\left(-b_{1}\tilde{Y}_{s}^{i}-(b_{2}\tilde{Y}_{s})\tilde{Y}^{i}_{s}\right) ds+\left(-f_{1}^{i}-f_{2}^{i} \tilde{Y}_{s}-\nabla X_{s}^{-1} (f_{3}^{i} (\nabla Z_{s}^{i})^{\mathsf{T}})\right) d s\\
& \quad +\nabla X_{s}^{-1}\nabla Z_{s}^{i} d W_{s}
\\
& = \left(-b_{1}\tilde{Y}_{s}^{i}-(b_{2}\tilde{Y}_{s})\tilde{Y}^{i}_{s}-f_{1}^{i}-f_{2}^{i}\tilde{Y}_{s}- f_{3}^{i}(\tilde{Z}_{s}^{i})^{\mathsf{T}}\right) ds+ \tilde{Z}_{s}^{i}d W_{s}
,\quad i = 1,2,\cdots,n,
\end{aligned} \label{compare BSDE}
\end{equation}
with $\tilde{Y}_{T} = h_{1}$.\\
    Without loss of any generality, it is sufficient to prove the argument for the case where assumptions (M1) and (M2)(i) hold. Now we introduce two n-dimensional BSDEs:
\begin{equation}
\begin{aligned}
\hat{Y}_{s}^{i} &= \int_{s}^{T}\hat{H}^{i}(r,\hat{Y}_{r},\hat{Z}_{r})dr-\int_{s}^{T}\hat{Z}_{r}^{i}dW_{r}\\
&=\int_{s}^{T}\left(\sum _{j=1}^{n}f_{2}^{ij}\hat{Y}_{r}^{j} +b_{1}\hat{Y}_{r}^{i} + \sum_{j=1}^{n}b_{2}^{j}\tilde{Y}_{r}^{j}\hat{Y}_{r}^{i}+ f_{3}^{i}(\hat{Z}_{r}^{i})^{\mathsf{T}}\right)dr-\int_{s}^{T}\hat{Z}_{r}^{i}dW_{r}
    ,\quad i=1,2,\cdots,n,
\end{aligned}\label{under BSDE}
\end{equation}
and
\begin{equation}
\begin{aligned}
\bar{Y}_{s}^{i} &=K + \int_{s}^{T}\bar{H}^{i}(r,\bar{Y}_{r},\bar{Z}_{r})dr-\int_{r}^{T}\bar{Z}^{i}_{r}dW_{r}\\
&=K+\int_{s}^{T}
    \left(K+ \sum_{j=1}^{n}K|\bar{Y}_{r}^{j}|+K|\bar{Y}_{r}^{i}|+f_{3}^{i }(\bar{Z}_{r}^{i})^{\mathsf{T}}\right)ds-\int_{s}^{T}\bar{Z}^{i}
    _{r}dW_{r}, \quad i=1,2,\cdots,n.\label{upper BSDE}
\end{aligned}
\end{equation}\\
It is easy to be verified that $(0,0)$ is the solution of the BSDE \eqref{under BSDE}, and $(y_{s}, 0)$ is the solution of BSDE \eqref{upper BSDE}, where $y_{s}$ is the solution of ODE \eqref{ODE} as in Lemma \ref{lemma 3.5}.\\
We are now going to show $0=\hat{Y}_{s} \le \tilde{Y}_{s} \le \bar{Y}_{s}=y_{s}$, under assumption (M1) and (M2)(i) for $s\in[t,T]$.
 The key tool is comparison theorem for multidimensional BSDEs (see \cite{hu2006comparison} for the Lipschitz case and \cite{luo2021comparison} for the quadratic case).\\
 The generator of  BSDE \eqref{compare BSDE} can be represented as the following form:
 \begin{equation}
   H^{i}(s,y,z)=f_{1}^{i}+\sum_{j=1}^{n} f_{2}^{i j} y^{j}+b_{1} y^{i}+\sum_{j=1}^{n} b_{2}^{j} \tilde{Y}_{s}^{j} y^{i}+f_{3}^{i}(z^{i})^{\mathsf{T}}, \quad i=1,2,\cdots,n. 
 \end{equation}
 Under assumptions (H),(A1) (or (A3)), $\tilde{Y}$, $b_{1}^{i},b_{2}^{i},f_{1}^{i},f_{2}^{i},f_{3}^{i}$ are bounded. Moreover under the conditions (M1) and (M2)(i), we get that $\hat{H}^{i}(s,y,z)\leq H^{i}(s,\bar{y},z)$ when $y^{i} = \bar{y}^{i}, y^{j} \le \bar{y}^{j},j\neq i$ and $h_{1}^{i} \ge 0$. By applying the comparison theorem \cite[Theorem 2.2]{hu2006comparison}, we get $0=\hat{Y}_{s}\le \tilde{Y}_{s}$, for all $s\in[t,T]$. As for under assumptions (H) and (A2), since $f_{3}^{i}\cdot W$ is a BMO martingale, by a slight modification of the proof of \cite[Theorem 2.3]{luo2021comparison}, we can also get $0=\hat{Y}_{s}\le \tilde{Y}_{s}$, for all $s\in[t,T]$. \\
In order to get the upper bound of $\tilde{Y}$, we represent the generator in another form:
\begin{equation*}
   H^{i\prime}(s,y,z)=f_{1}^{i}+\sum_{j=1}^{n} f_{2}^{i j} y^{j}+b_{1} y^{i}+\sum_{j=1}^{n} b_{2}^{j} \tilde{Y}_{s}^{j} \tilde{Y}_{s}^{i}+f_{3}^{i}(z^{i})^{\mathsf{T}}, \quad i=1,2,\cdots,n.  
\end{equation*}
Noting the non-negativity of $\tilde{Y}$ and $b_{2}^{j} \le 0$ under condition (M2)(i), combining with condition (M1), for $y,\bar{y}$ satisfying $y^i=\bar{y}^i, y^j \leq \bar{y}^j, j \neq i$, we have
\begin{equation}
\begin{aligned}
H^{i \prime}(s, y, z) & \leq f_1^i+\sum_{j=1}^n f_2^{i j} \bar{y}^j+b_1 \bar{y}^i+f_3^i(z^i)^{\mathsf{T}} \\
& \leq K+\sum_{j=1}^n K\left|\bar{y}^j\right|+K\left|\bar{y}^i\right|+f_3^i(z^i)^{\mathsf{T}} \\
& =\bar{H}^i(s, \bar{y}, z), \quad i=1,2, \cdots, n.
\end{aligned} 
\end{equation}
 Combined with the terminal condition that $h_1^i \leq K$, using similar argument as above, we obtain $\tilde{Y}_s \leq \bar{Y}_s=y_{s}$, for all $s \in[t, T]$. For the case where assumptions (M1) and (M2)(ii) hold, we can similarly prove 
 $-y_{t}\leq\tilde{Y}_{t}\leq 0$, which concludes the proof.
\end{proof}
\begin{remark}
We should point out that for the quadratic case, the time duration $\delta$ depends on  $K,\lambda $, where $\lambda $ is the bound of the terminal condition. Therefore, in addition to  obtaining the uniform Lipschitz constant of the decoupling field, it is critical to obtain the uniform bound of the value process $Y$ in order to obtain global existence, which will be tackled rigorously in Theorem \ref{Theorem:global}.
\end{remark}
\indent We are now ready to state our second main result, which gives the existence and uniqueness of solution for FBSDE \eqref{FBSDE} for arbitrarily large $T$.
\begin{Theorem}
Under assumptions (H), (M1),(M2) and if one of assumptions (A1)-(A3) holds, for any $t\in[0,T]$ and $x\in\mathbb{R}$, the FBSDE \eqref{small FBSDE} admits a unique solution. Further the random function defined on $\Omega \times [0,T]\times \mathbb{R}$ by
\begin{equation*}
    u(\omega,t,x) = Y_{t}^{t,x}(\omega)
\end{equation*}
satisfies for any $t\in[0,T]$ and $s\in[t,T]$,
\begin{equation*}
    Y_{s}^{t,x} = u(s,X^{t,x}_{s}), P\text{-}a.s.
\end{equation*}
and
    \begin{equation*}
\left|\frac{u(t,x_{1})-u(t,x_{2})}{x_{1}-x_{2}}\right|\leq C,
    \end{equation*}
for some constant $C$ only depending on $n,K,T$ and for all $x_1,x_2\in\mathbb{R}$ and $x_1\neq x_2$. In particular, FBSDE \eqref{FBSDE} admits a unique solution.
\label{Theorem:global}
\end{Theorem}
\begin{proof} 
First, applying lemma \ref{lemma 3.5}, there exists a constant $C$, which depends on $n,K,T$ such that
\begin{equation}
     |y_{s}| \leq C \label{BOUND}
\end{equation}
for all $0 \leq s \leq T$.\\
\indent Then, let's consider Lipschitz and super-quadratic cases. For any $t\in[0,T]$, let $\delta>0$ be the constant determined by $C$ and $t=t_{0}<$ $\cdots<t_{m}=T$ be a partition of $[t, T]$ such that $t_{i}-t_{i-1} \leq \delta$, $i=1, \cdots, m$. We first consider FBSDE \eqref{small FBSDE} on $\left[t_{m-1}, T\right]$. Since $ 0 \leq h_{1} \leq y_{T}$, we see that the Lipschitz constant of the terminal condition $h$ is less than  $C$, then by Theorem \ref{Theorem :small duration}, there exists a function $ u(t_{m-1},\cdot)$ which satisfies 
\begin{equation}\label{eq:Lip_u}
    \left|\frac{u(t_{m-1},x_{1})-u(t_{m-1},x_{2}) }{x_{1}-x_{2}}\right|\leq |y_{t_{m-1}}|\leq C.
      \end{equation}
 Repeating this procedure backwardly finitely many times, we can find $u(t_{i},\cdot)$ satisfying 
\begin{equation}
\left|\frac{u(t_{i},x_{1})-u(t_{i},x_{2}) }{x_{1}-x_{2}}\right|\leq |y_{t_{i}}|\leq C, \quad i=0,1,2,\cdots,m.\label{n ut}
  \end{equation}
\indent As for quadratic case, by constructing a ODE similar with the one in \cite[Theorem 2.3]{hu2016}, there exists a constant $\tilde{\lambda}$ which depends only on $K,n,T,\lambda$ such that $|h(T,\cdot)|\le \tilde{\lambda}$. Let $\delta>0$ be the constant determined by $C$ and $\tilde{\lambda}$,  and $t=t_{0}<$ $\cdots<t_{m}=T$ be a partition of $[t, T]$ such that $t_{i}-t_{i-1} \leq \delta$, $i=1, \cdots, m$, then we can find a function $u(t_{m-1},\cdot)$  satisfies \eqref{eq:Lip_u} following the same argument with Lipchitz and super-quadratic cases. Using the similar argument with \cite[Theorem 2.3]{hu2016}, we can prove $|u(t_{m-1},\cdot)|\le \tilde{\lambda}$. Repeating the preceding process, we can find $u(t_{i},\cdot)$ satisfying \eqref{n ut}.

Recalling \eqref{n ut}, it follows from Theorem \ref{Theorem :small duration} that FBSDE \eqref{small FBSDE}  admits a unique solution on $[t,t_1]$ with initial condition $x$ and terminal function $u(t_{1},\cdot)$. Recursively, for $i=1,2,3,\cdots,m-1$,  FBSDE \eqref{small FBSDE}  admits a unique solution on $[t_i,t_{i+1}]$ with initial condition $X_{t_i}$ and terminal function $u(t_{i+1},\cdot)$. The decoupling field property \eqref{decoupling field} ensures the small duration solutions can be patched together. Thus we get the existence of a solution for FBSDE \eqref{small FBSDE} on $[t,T]$. The uniqueness follows immediately from the uniquness on each interval. In particular, the well-posedness of FBSDE \eqref{FBSDE} is obtained with $t=0$. The proof is complete.
\end{proof}
\begin{remark}
Antonelli and Hamad{\`e}ne \cite{antonelli2006existence} prove the global existence of solution for FBSDE \eqref{FBSDE}  under the monotonicity condition that the coefficients $b$ is increasing in $y$ and $f$ is increasing $x$. Recently, this approach is further extended by Chen and Luo \cite{chen2021existence} to study multi-dimensional coupled FBSDEs with diagonally quadratic generators . In both papers, they fail to ensure the uniqueness of the solution. Compared with them, we can establish the uniqueness of solution for FBSDE \eqref{FBSDE} under the opposite monotonicity of $b$ in $y$ and $f$ in $x$. 
\end{remark}
\begin{remark}
Luo and Tangpi \cite{luo2017solvability} obtain local solvability for diagonally quadratic FBSDEs. In Markovian setting, Kupper, Luo and Tangpi \cite{kupper2019} give global solvability for super-quadratic FBSDEs, while Jackson \cite{jackson2021global} considers the global solvability for quadratic FBSDEs. Compared with \cite{luo2017solvability}, we obtain global solution. Compared with \cite{jackson2021global} and \cite{kupper2019}, our results do not require the non-degeneracy of the diffusion process.
\label{quadratic remark}
\end{remark}
\begin{remark}
Compared with \cite{peng1999}, we propose a different monotonicity condition, which allows us to deal with some FBSDEs which can not be solved by Peng and Wu \cite{peng1999}. Let us illustrate it with the following example.
\begin{equation}
\begin{aligned}
X_{t}&=x_{0}+\int_{0}^{t}(-Y_{s})ds+\int_{0}^{t}\sigma_{s}dW_{s},\\
Y_{t}&=X_{T}+\int_{t}^{T}(X_{s}-Y_{s}-Z_{s})ds-\int_{t}^{T}Z_{s}dW_{s},
\end{aligned}\label{example}
\end{equation}
where $X\in\mathbb{R}, Y\in\mathbb{R}, Z\in\mathbb{R}$. It is obvious that above FBSDE satisfies assumptions (H),(A1),(M1) and (M2). According to Theorem \ref{Theorem:global} , FBSDE \eqref{example} admits a unique solution, whereas the monotonicity assumption of (H2.3) in \cite{peng1999} fails. Indeed the monotonicity for $\Phi$ still holds true,
\begin{equation*}
    \langle \Phi(x)-\Phi(\bar{x}),x-\bar{x}\rangle = (x-\bar{x})^{2}\ge 0,
\end{equation*}
but we have 
\begin{equation*}
    \langle A(t,u)-A(t,\bar{u}),u-\bar{u}\rangle =G[ -(x-\bar{x})^{2}-(y-\bar{y})^{2}+(x-\bar{x})(y-\bar{y})+(x-\bar{x})(z-\bar{z})],
\end{equation*}
where $A(t,u)= \left(\begin{array}{c}
-G^T f \\
G b \\
G \sigma
\end{array}\right)(t, u),u=(x,y,z),\bar{u}=(\bar{x},\bar{y},\bar{z})$.\\
\indent In general, we can not find a $G$ such that monotonicity assumption holds since the existence of the intersection of $x$ and $z$. \label{continuation remark}
\end{remark}
\section{Some extensions in Lipschitz case}
\subsection{A more general Lipschitz FBSDE}
In this subsection, we provide some extensions of our global solvability results. Indeed, by restricting to Lipschitz generators, our approach still works to solve  a more general FBSDE \eqref{FBSDE5} globally, where the forward diffusion is allowed to depend on $X$ and $Y$, that is we consider the following FBSDE  
\begin{equation}
\left\{\begin{array}{l}
X_t=x+\int_0^t b\left(s, X_s, Y_s\right) d s+\int_0^t \sigma\left(s, X_s, Y_s\right) d W_s, \\
Y_t^i=h^i\left(X_T\right)+\int_t^T f^i\left(s, X_s, Y_s, Z_s^i\right) d s-\int_t^T Z_s^i d W_s, i=1,2, \cdots, n,
\end{array}\right.\label{FBSDE5}
\end{equation}
where $X \in \mathbb{R}, Y \in \mathbb{R}^n, Z \in \mathbb{R}^{n \times d}$.\\
\indent In this section, given a positive constant $K$,  we make the following assumptions:
\vspace{0.8em}
\begin{itemize}
    \item [\textbf{(B1)}]
    \begin{itemize}
        \item [(i)]  $b: \Omega \times[0, T] \times \mathbb{R} \times \mathbb{R}^{n} \rightarrow \mathbb{R},f:\Omega \times [0, T] \times \mathbb{R} \times \mathbb{R}^{n} \times \mathbb{R}^{n \times d} \rightarrow \mathbb{R}^{n}$,  $\sigma: \Omega \times [0,T] \times \mathbb{R} \times \mathbb{R}^{n} \rightarrow \mathbb{R}^{d}$ are progressive measurable, $h: \Omega \times \mathbb{R} \rightarrow \mathbb{R}^{n}$ is $\mathcal{F}_{T}$-measurable, $f^{i}(t,x,y,z) = f^{i}(t,x,y,z^{i})$ for $i=1,2,\cdots,n$, and 
\begin{equation*}
\begin{aligned}
\left|b(t,x, y)-b\left(t,x^{\prime}, y^{\prime}\right)\right| & \leq K(\left|x-x^{\prime}\right|+\left|y-y^{\prime}\right|),\\
\left|f(t,x, y, z)-f\left(t,x^{\prime}, y^{\prime}, z^{\prime}\right)\right| & \leq K(\left|x-x^{\prime}\right|+|y-y^{\prime}|+|z-z^{\prime}|),\\
\left|h(x)-h\left(x^{\prime}\right)\right| &\leq K\left|x-x^{\prime}\right|,\\
  |\sigma(t,x,y)-\sigma(t,x^{\prime},y^{\prime})| &\leq K(|x-x^{\prime}|+|y-y^{\prime}|)\\
 \end{aligned}
\end{equation*}
for all $x, x^{\prime} \in \mathbb{R}, y,y^{\prime} \in \mathbb{R}^{n}$ and $z,z^{\prime} \in \mathbb{R}^{n \times d}$.
   \item[(ii)]
   The following integrability condition holds
   \begin{equation*}
 \mathbb{E}\left\{\left(\int_0^T[|b(t,0,0)|+|f(t,0,0,0)|]dt\right)^{2} +\int_{0}^{T} |\sigma(t, 0,0)|^{2}d t+|h(0)|^2\right\}<\infty.     
 \end{equation*}
\end{itemize}
\end{itemize}
\indent Next we give some notations used in this section. When $b,f,h,\sigma$ satisfy assumption (B1), for $\left(x_{1}, y_{1}, z_{1}\right),\left(x_{2}, y_{2}, z_{2}\right)\in\mathbb{R}\times\mathbb{R}^n\times\mathbb{R}^{n\times d}$, let $\theta_1:=\left(x_{1}, y_{1}, z_{1}\right),\theta_2:=\left(x_{2}, y_{2}, z_{2}\right)$ , $b_{j}(t,\theta_{1},\theta_{2})$, $j=1,2$, $f_{j}(t,\theta_{1},\theta_{2})$, $j=1,2,3$ and $h_{1}(x_{1},x_{2})$ are defined by \eqref{quadratic notation}. For $i=1,2 \cdots,d$, $j=1,2,\cdots,n$,
\begin{equation}
\begin{aligned}
       \sigma_{1}^{i}(t,\theta_{1},\theta_{2}) &\triangleq \frac{\sigma^{i}(t,x_{1},y_{1})-\sigma^{i}(t,x_{2},y_{1})}{x_{1}-x_{2}},\\
       \sigma_{2}^{ij}(t,\theta_{1},\theta_{2}) &\triangleq\frac{\sigma^{i}(t,x_{2},y_{2}^{(1,j-1)},y_{1}^{(j,n)})-\sigma^{i}(t,x_{2},y_{2}^{(1,j)},y_{1}^{(j+1,n)})}{y_{1}^{j}-y_{2}^{j}},
\end{aligned}\label{quadratic sigma}
\end{equation}
and  $\sigma_{2}(t,\theta_{1},\theta_{2}) = (\sigma_{2}^{1},\sigma_{2}^{2},\cdots,\sigma_{2}^{d})^{\mathsf{T}}(t,\theta_{1},\theta_{2})(t,\theta_{1},\theta_{2})$, $\sigma_{1}(t,\theta_{1},\theta_{2}) = (\sigma_{1}^{1},\sigma_{1}^{2},\cdots,\sigma_{1}^{d})^{\mathsf{T}}(t,\theta_{1},\theta_{2})$, where $\sigma_{2}^{i}(t,\theta_{1},\theta_{2}) = (\sigma_{2}^{i1},\sigma_{2}^{i2},\cdots,\sigma_{2}^{in})^{\mathsf{T}}(t,\theta_{1},\theta_{2})$.\\
Then we impose the following monotonicity conditions.
\begin{itemize}
\item[\textbf{(M3)}]
$b_{2}(t,\theta_{1},\theta_{2}),f_{3}^{i}(t,\theta_{1},\theta_{2}),\sigma_{2}(t,\theta_{1},\theta_{2})$ are defined by \eqref{quadratic notation} and \eqref{quadratic sigma}, for $t\in[0,T]$, one of the following two cases holds
    \begin{itemize}
        \item[(i)] 
    For $1 \leq i \leq n$, we have
\begin{equation}
  f^{i}(t,x,y,z^{i}) \leq f^{i}(t,\bar{x},y,z^{i}), \quad h^{i}(x)\leq h^{i}(\bar{x}),
\end{equation} 
 for $(x,y,z), (\bar{x},y,z)\in \mathbb{R}\times \mathbb{R}^{n} \times \mathbb{R}^{n\times d}$ satisfying $x\leq \bar{x}$,
and
\begin{equation*}
b_{2}(t,\theta_{1},\theta_{2})+f_{3}^{i}(t,\theta_{1},\theta_{2}) \sigma_{2}(t,\theta_{1},\theta_{2}) \leq 0,
\end{equation*}
for all $\theta_{1},\theta_{2}\in \mathbb{R} \times \mathbb{R}^{n}\times \mathbb{R}^{n\times d}$.
\item[(ii)] 
 For $1 \leq i \leq n$, we have
\begin{equation}
  f^{i}(x,y,z^{i}) \geq f^{i}(\bar{x},y,z^{i}), \quad h^{i}(x)\geq h^{i}(\bar{x}),
\end{equation} 
 for $(x,y,z), (\bar{x},y,z)\in \mathbb{R}\times \mathbb{R}^{n} \times \mathbb{R}^{n\times d}$ satisfying $x\leq \bar{x}$,
and
\begin{equation*}
b_{2}(t,\theta_{1},\theta_{2})+f_{3}^{i}(t,\theta_{1},\theta_{2}) \sigma_{2}(t,\theta_{1},\theta_{2}) \geq 0,
\end{equation*}
for all $\theta_{1},\theta_{2}\in \mathbb{R} \times \mathbb{R}^{n}\times \mathbb{R}^{n\times d}$.
\end{itemize}
\end{itemize}
\indent The following theorem ensures global existence and uniqueness of solution for FBSDE \eqref{FBSDE5}, which extends some results of Zhang
\cite{zhang2017wellposedness}. Indeed, the global solvability of FBSDE \eqref{FBSDE5} can be deduced from \cite{zhang2017wellposedness} by assuming $b_{2}=0 $. However, FBSDEs arising from optimal control problems ususally does not satisfy this condition. By using different argument, we are able to deal with this case by additionally imposing some monotonicity conditions.
\begin{Theorem}
Suppose assumptions (B1), (M1), (M3) hold, then for any $t\in[0,T]$ and $x\in\mathbb{R}$, the following FBSDE
\begin{equation}
\left\{\begin{array}{l}
X^{t,x}_{s}=x+\int _{t}^{s} b\left(r, X^{t,x}_{r}, Y^{t,x}_{r}\right) d r+\int _{t}^{s} \sigma(r,X_{r}^{t,x},Y_{r}^{t,x}) d W_{r}, \\
Y_{s}^{t,x,i}=h^{i}\left(X^{t,x}_{T}\right)+\int_{s}^{T} f^{i}\left(r, X^{t,x}_{r}, Y^{t,x}_{r}, Z_{r}^{t,x,i}\right) d r-\int_{s}^{T} Z_{r}^{t,x,i}d W_{r}, \quad i=1,2,\cdots,n,
\end{array}\right.\label{small lip FBSDE}
\end{equation}
has a unique solution $(X^{t,x},Y^{t,x},Z^{t,x})\in \mathcal{S}^{2}(\mathbb{R}) \times \mathcal{S}^{2}(\mathbb{R}^{n}) \times \mathcal{H}^{2}\left(\mathbb{R}^{n \times d}\right)$on $[t,T]$. Further the random function $u$ defined on $\Omega\times[0,T]\times\mathbb{R}$ by 
\begin{equation*}
    u(\omega,t,x) = Y_{t}^{t,x}(\omega),
\end{equation*}
satisfies for any $t\in[0,T]$ and $s\in[t,T]$,
\begin{equation*}
    Y_{s}^{t,x} = u(s,X^{t,x}_{s}), ~~P\text{-}a.s.
    \end{equation*}
    and
    \begin{equation*}
\left|\frac{u(t,x_{1})-u(t,x_{2})}{x_{1}-x_{2}}\right|\leq C
    \end{equation*}
   for some constant $C$ only depending on $n,K,T$ and for all $x_1,x_2\in\mathbb{R}$ and $x_1\neq x_2$. In particular, FBSDE \eqref{FBSDE5} has a unique solution in $\mathcal{S}^{2}(\mathbb{R}) \times \mathcal{S}^{2}(\mathbb{R}^{n}) \times \mathcal{H}^{2}\left(\mathbb{R}^{n \times d}\right)$.
\label{Lipschitz corollary}
\end{Theorem}
\begin{proof}
First, it follows from \cite{ma1999book} that there exists a constant $\delta$, which depends only on the Lipschitz constant $K$, such that for any $t\in[T-\delta,T]$ and $x\in\mathbb{R}$, the FBSDE \eqref{small lip FBSDE} has a unique solution $(X^{t,x},Y^{t,x},Z^{t,x}) \in \mathcal{S}^{2}(\mathbb{R}) \times \mathcal{S}^{2}\left(\mathbb{R}^{n}\right) \times \mathcal{H}^{2}\left(\mathbb{R}^{n\times d}\right)$.\\
 \indent Now for any $t\in[T-\delta,T]$ and any $x_{1},x_{2}\in \mathbb{R}$ satisfying $x_{1} \neq x_{2}$, we denote $\Theta^{i}=(X^{t,x_{i}},Y^{t,x_{i}},Z^{t,x_{i}})$ for $i=1,2$. Using the same notations with \eqref{theta}, it is easy to verify that $(\nabla X, \nabla Y, \nabla Z)$ satisfies\\
\begin{equation*}
\left\{\begin{array}{rlr}
\nabla X_{s}= & 1+\int_{t}^{s}\left(b_{1}(r) \nabla X_{r}+b_{2}(r)\nabla Y_{r}\right) d r+\int_{t}^{s}(\nabla X_{r}\sigma_{1}^{\mathsf{T}}(r)+\nabla Y_{r}^{\mathsf{T}}\sigma_{2}^{\mathsf{T}}(r))dW_{r}, \\
\nabla Y_{s}= & h_{1}\nabla X_{T}+\int_{s}^{T}\left(f_{1}(r) \nabla X_{r}+f_{2}(r)\nabla Y_{r}+f_{3}(r) \cdot \nabla Z_{r}\right) dr -\int_{s}^{T} \nabla Z_{r} d W_{r},
\end{array}\right.
\end{equation*}
where $h_{1}\triangleq h_{1}^{i}(X_{T}^{1},X_{T}^{2})$, $ \phi_{j}(r)\triangleq \phi_{j}(r,\Theta^{1}_{t},\Theta^{2}_{t})$, $\phi = b,f$, $j=1,2,3$ respectively, $\sigma_{j}(r) \triangleq \sigma_{j}(r,\Theta^{1}_{t},\Theta^{2}_{t})$, $j=1,2$ respectively. We note here that $b_{1}(r),b_{2}(r),f_{1}(r),f_{2}(r),f_{3}(r)$ are $\mathcal{F}_{r}$ adapted processes and $h_{1}$ is an $\mathcal{F}_{T}$-measurable random variable, and we will omit the time variable in the following. 

 Now, define a random field as $u(\omega,t,x)= Y^{t,x}_{t}(\omega)$.
Following similar argument in Theorem \ref{Theorem :small duration}, by choosing a smaller $\delta$ (only depending on $K$) if necessary, we have $|\nabla Y_{s}|\leq C|\nabla X_{s}|$ and further $Y^{t,x}_{s} = u(s,X^{t,x}_{s}), s\in[t,T],P$-a.s.

Now we prove $\nabla X_{s} >0$ for all $s \in [t,T]$.
If we let $\tau = T \land \inf \{s \ge t:\nabla X_{s}=0\}$.
To this end note that $\nabla X_{s}$ satisfies, on $[t,\tau)$, the linear SDE:
\begin{equation}
  d \nabla X_{s}=\left(b_{1} \nabla X_{s}+b_{2} \nabla Y_{s}\right) d s+\left(\nabla X_{s}\sigma_{1}^{\mathsf{T}}+\nabla Y_{s}^{\mathsf{T}}\sigma_{2}^{\mathsf{T}} \right) d W_{s}. \label{4.5}  
\end{equation}
Denote
\begin{equation*}
    \tilde{Y}_{s}= \nabla X_{s}^{-1}\nabla Y_{s},
\end{equation*}
we obtain
\begin{equation*}
    d \nabla X_{s}=\nabla X_{s}\left[\left(b_{1}+b_{2}\tilde{Y}\right) ds+\left(\sigma_{1}^{\mathsf{T}}+\tilde{Y}^{\mathsf{T}}\sigma_{2}^{\mathsf{T}}\right)dW_{s}\right],\quad s\in [t,\tau),
\end{equation*}
then 
\begin{equation}
 \nabla X_{s}=\exp \left[ \int_{t}^{s}\left(b_{1}+b_{2} \tilde{Y} \right) dr-\frac{1}{2}|\sigma_{1}^{\mathsf{T}}+\tilde{Y}^{\mathsf{T}}\sigma_{2}^{\mathsf{T}}|^{2}dr +\int_{t}^{s}\left(\sigma_{1}^{\mathsf{T}}+\tilde{Y}^{\mathsf{T}}\sigma_{2}^{\mathsf{T}} \right) d W_{r}\right], \quad s\in [t,\tau).  \end{equation}
Indeed, since $b_{1},b_{2},\sigma_{1},\sigma_{2},\tilde{Y}_{t}$ are all uniformly bounded, which implies that $\nabla X_{t}$ stays positive and then $\tau = T$.\\
Moreover we have
\begin{equation*}
\begin{aligned}
  d\nabla X_{s}^{-1} &=-\nabla X_{s}^{-2}d\nabla X_{s}+\nabla X ^{-3}_{s}d\langle\nabla X_{s}\rangle\\
  &=-\nabla X_{s}^{-2}\left[(b_{1}\nabla X_{s} + b_{2} \nabla Y_{s})dt + (\nabla X_{s}\sigma_{1}^{\mathsf{T}}+\nabla Y_{s}^{\mathsf{T}}\sigma_{2}^{\mathsf{T}})dW_{s}\right]\\
  &\quad +\nabla X_{s}^{-3}\left[|\nabla X_{s}\sigma_{1}^{\mathsf{T}}+\nabla Y_{s}^{\mathsf{T}}\sigma_{2}^{\mathsf{T}}|^{2}ds\right]\\
  &=-\nabla X_{s}^{-1}\left[(b_{1}+b_{2} \tilde{Y}_{s}-|\sigma_{1}^{\mathsf{T}}+ \tilde{Y}_{s}^{\mathsf{T}}\sigma_{2}^{\mathsf{T}}|^{2})ds+(\sigma_{1}^{\mathsf{T}}+ \tilde{Y}_{s}^{\mathsf{T}}\sigma_{2}^{\mathsf{T}})dW_{s}\right],\\
  \end{aligned}
\end{equation*}
and 
\begin{equation*}
    d\nabla Y_{s}^{i} = (-f_{1}^{i} \nabla X_{s}-f_{2}^{i}\nabla Y_{s}-f_{3}^{i}(\nabla Z_{s}^{i})^{\mathsf{T}})ds + \nabla Z_{s}^{i}dW_{s},\quad i=1,2\cdots,n.
\end{equation*}
The dynamics of $\tilde{Y}_{s}$ are now deduced from those of $\nabla Y_{s}$ and $[\nabla X_{s}]^{-1}$ using the product rule. For all $s \in[t, T]$, it holds
\begin{equation}
\begin{aligned}
d\tilde{Y}_{s}^{i}&=\nabla Y_{s}^{i} d \nabla X_{s}^{-1}+\nabla X_{s}^{-1} d \nabla Y_{s}^{i}+d\left\langle\nabla X_{s}^{-1}, \nabla Y_{s}\right\rangle \\ 
&=-\tilde{Y}_{s}^{i}\left[(b_{1}+b_{2} \tilde{Y}_{s}-|\sigma_{1}^{\mathsf{T}}+\tilde{Y}_{s}^{\mathsf{T}}\sigma_{2}^{\mathsf{T}}|^{2})ds+(\sigma_{1}^{\mathsf{T}}+ \tilde{Y}_{s}^{\mathsf{T}}\sigma_{2}^{\mathsf{T}})dW_{s}\right]\\
& \quad +(-f_{1}^{i}-f_{2} \tilde{Y}_{s}-\nabla X_{s}^{-1}(f_{3}^{i}(\nabla Z_{s}^{i})^{\mathsf{T}}))dt+\nabla X_{s}^{-1}\nabla Z_{s}^{i} dW_{s}\\
&\quad -\langle \nabla X_{s}^{-1}(\sigma_{1}^{\mathsf{T}}+ \tilde{Y}_{s}^{\mathsf{T}}\sigma_{2}^{\mathsf{T}}),\nabla Z_{s}^{i}\rangle ds\\
&= [-b_{1}\tilde{Y}_{s}^{i}-b_{2} \tilde{Y}_{s}\tilde{Y}_{s}^{i}-f_{1}^{i}-f_{2}\tilde{Y}_{s}-\nabla X_{s}^{-1}(f_{3}^{i}(\nabla Z_{s}^{i})^{\mathsf{T}})+|\sigma_{1}^{\mathsf{T}}+\tilde{Y}_{s}^{\mathsf{T}}\sigma_{2}^{\mathsf{T}}|^{2}\tilde{Y}_{s}^{i}\\
&\quad -\langle \nabla X_{s}^{-1}(\sigma_{1}^{\mathsf{T}}+\tilde{Y}_{s}^{\mathsf{T}}\sigma_{2}^{\mathsf{T}}),\nabla Z_{s}^{i} \rangle]ds+\left[\nabla X_{s}^{-1} \nabla Z_{s}^{i}-(\sigma_{1}^{\mathsf{T}}+ \tilde{Y}_{s}^{\mathsf{T}}\sigma_{2}^{\mathsf{T}})\tilde{Y}_{s}^{i}\right]dW_{s}.
\end{aligned} \label{dy}
\end{equation}
Denote
\begin{equation}
    \tilde{Z}_{s}^{i} \triangleq \nabla X_{s}^{-1}\nabla Z_{s}^{i}-(\sigma_{1}^{\mathsf{T}}+\tilde{Y}_{s}^{\mathsf{T}}\sigma_{2}^{\mathsf{T}})\tilde{Y}_{s}^{i}. \label{dz}
\end{equation}
By substituting \eqref{dz} to \eqref{dy}, we obtain
\begin{equation}
\begin{aligned}
 d\tilde{Y}_{s}^{i} & =- f_{1}^{i}-f_{2}\tilde{Y}_{s}-b_{1}\tilde{Y}_{s}^{i}-b_{2}\tilde{Y}_{s}\tilde{Y}_{s}^{i}-f_{3}^{i}(\sigma_{1}+\sigma_{2}\tilde{Y}_{s})\tilde{Y}_{s}^{i}\\ 
    &\quad -\langle f_{3}^{i}+\sigma_{1}^{\mathsf{T}}+\tilde{Y}_{s}^{\mathsf{T}}\sigma_{2}^{\mathsf{T}},\tilde{Z}_{s}^{i} \rangle ds + \tilde{Z}_{s}^{i}dW_{s}\quad i=1,2,\cdots,n.
\end{aligned}\label{zong}
\end{equation}
Without loss of any generality, it is sufficient to prove the argument for the case where assumptions (M1) and (M3)(i) hold. Now we introduce the following two multi-dimensional BSDEs:
\begin{equation}
  \begin{aligned}
  \hat{Y}_{s}^{i} &= \int_{s}^{T}
    \hat{H}^{i}(r,\hat{Y}_{r},\hat{Z}_{r})dr-\int_{s}^{T}\hat{Z}_{r}^{i}dW_{r} \\
    &=\int_{s}^{T} ( f_{2}\hat {Y}_{r}+b_{1}\hat{Y}_{r}^{i}+b_{2} \tilde{Y}_{r}\hat{Y}_{r}^{i}+f_{3}^{i}\sigma_{1}\hat{Y}_{r}^{i}+f_{3}^{i}\sigma_{2}\tilde{Y}_{r}\hat{Y}_{r}^{i}\\
    &\quad +\langle f_{3}^{i}+\sigma_{1}^{\mathsf{T}}+ \tilde{Y}_{r}^{\mathsf{T}}\sigma_{2}^{\mathsf{T}},\hat{Z}_{r}^{i} \rangle ds-\int_{s}^{T}\hat{Z}_{r}^{i}dW_{r}, i =1,2,\cdots,n,
\end{aligned} \label{lips under bsde}  
\end{equation}
and
\begin{equation}
\begin{aligned}
\bar{Y}_{s}^{i} &= K +\int_{s}^{T}
   \bar{H}^{i}(r,\bar{Y}_{r},\bar{Z}_{r})dr -\int_{s}^{T}\bar{Z}_{r}^{i}dW_{r}\\
   &=  K +\int_{s}^{T}K +K|\bar{Y}^{i}_r|+\sum_{j=1}^{n}K |\bar{Y}^{j}_r|+dK^{2}|\bar{Y}^{i}_r|\\
   &\quad +\langle f_{3}^{i}+\sigma_{1}^{\mathsf{T}}+\tilde{Y}_{r}^{\mathsf{T}}\sigma_{2}^{\mathsf{T}},\bar{Z}_{r}^{i} \rangle dr-\int_{s}^{T}\bar{Z}_{r}^{i}dW_{r}, \quad i=1,2,\cdots,n. \label{lips upper BSDE} 
\end{aligned}
\end{equation}
It is obvious that $(0,0)$ is the solution of BSDE \eqref{lips under bsde} and $(\bar{y}_{s}, 0)$ is the solution of BSDE \eqref{lips upper BSDE}, where $\bar{y}_{s}$ is the unique solution of ODE on $[0,T]$
\begin{equation*}
    \bar{y}_{s} = H + \int_{s}^{T} (A\bar{y}_{r}+(K+dK^2)\bar{y}_{r}+B)dr,
\end{equation*}
where $H, B\in \mathbb{R}^n, A \in \mathbb{R}^{n\times n}$ and all the components of $H, B$ and $A$ are $K$. We are now going to show $\hat{Y}_{s} \le \tilde{Y}_{s} \le \bar{Y}_{s}$, for $s\in[t,T]$. \\
The generator of  BSDE \eqref{zong}  can be represented as the following form:
\begin{equation}
\begin{aligned}
   H^{i}(s,y,z)=& f_{1}+f_{2}y+b_{1}y^{i}+b_{2}\tilde{Y}_{s} y^{i}+f_{3}^{i}\sigma_{1}y^{i}+f_{3}^{i}\sigma_{2}\tilde{Y}_{s}y^{i}\\
   &\langle f_{3}^{i}+\sigma_{1}^{\mathsf{T}}+\tilde{Y}_{s}^{\mathsf{T}}\sigma_{2}^{\mathsf{T}},z^{i} \rangle\quad i=1,2,\cdots,n.
\end{aligned}
\end{equation}
Indeed, we can use the comparison theorem for BSDEs since $\tilde{Y},f_{i},b_{i},\sigma_{i}$ are bounded on $[t,T]$ under assumption (B2).  Moreover under the conditions (M1) and (M3)(i), we get that $\hat{H}^{i}(s,y,z)\leq H^{i}(s,\bar{y},z)$ when $y^{i} = \bar{y}^{i}, y^{j} \le \bar{y}^{j},j\neq i$ and $h_{1}^{i} \ge 0$, it follows from comparison theorem for multi-dimensional BSDEs that $ 0=\hat{Y}_{s}\leq \tilde{Y}_{s},  \text{ for } s \in [t,T]$.\\
In order to get the upper bound of $\tilde{Y}$, we represent the generator in another form:
\begin{equation}
    \begin{aligned}
    H^{\prime i}(s,y,z) &= f_{1}+f_{2}y+b_{1}y^{i}+f_{3}^{i}\sigma_{1}y^{i}+(b_{2}+f_{3}^{i}\sigma_{2})\tilde{Y}_{s}\tilde{Y}_{s}^{i}\\
   &\langle f_{3}^{i}+\sigma_{1}^{\mathsf{T}}+\tilde{Y}_{s}^{\mathsf{T}}\sigma_{2}^{\mathsf{T}},z^{i} \rangle, \quad i=1,2,\cdots,n.\\
    \end{aligned}
\end{equation}
Noting the non-negativity of $\tilde{Y}$ and $b_{2}+f_{3}^{i}\sigma_{2}\leq 0$ under condition (M3)(i), combining with condition (M1), for $y,\bar{y}\in \mathbb{R}^{n}$ satisfying $y^{i}=\bar{y}^{i}$ and $y^{j}\leq\bar{y}^{j},j\neq i$, we get 
\begin{equation*}
\begin{aligned}
H^{\prime i}(s,y,z) &\leq f_{1} +f_{2}\bar{y}+b_{1}\bar{y}^{i}+ f_{3}^{i} \sigma_{1}\bar{y}^{i}+\langle f_{3}^{i}+\sigma_{1}^{\mathsf{T}}+\tilde{Y}_{s}^{\mathsf{T}}\sigma_{2}^{\mathsf{T}},z^{i} \rangle\\
& \le K +\sum_{j=1}^{n}K |\bar{y}^{j}|+K|\bar{y}^{i}|+dK^{2}|\bar{y}^{i}|+\langle f_{3}^{i}+\sigma_{1}^{\mathsf{T}}+\tilde{Y}_{s}^{\mathsf{T}}\sigma_{2}^{\mathsf{T}},z^{i} \rangle\\
&=\bar{H}^{i}(s,\bar{y},z),\quad i=1,2,\cdots,n.
\end{aligned}
\end{equation*}
Combined with $\tilde{Y}_{T} = h_{1}\le K= \bar{Y}_{T}$, it follows from similar argument as above that $\tilde{Y}_{s} \le \bar{Y}_{s}=\bar{y}_{s}, \text{ for } s\in [t,T]$. For the case where
conditions (M1) and (M3)(i) hold, we can obtain $-\bar{y}_{s} \leq \tilde{Y}_{s}\leq 0$, for any $s\in[t,T]$ similarly.\\
 Finally, with above small duration results, we can show for any $t\in[0,T]$ and $x\in\mathbb{R}$, the existence and uniqueness of solutions of FBSDEs \eqref{small lip FBSDE} with similar argument used in Theorem \ref{Theorem:global}. The proof is now complete.
\end{proof}
\subsection{\texorpdfstring{$L^{p}$}.-solution and \texorpdfstring{$L^{p}$}. estimates of FBDSEs}

\indent $L^{p}$-theory of FBSDEs has important application in stochastic optimal control theory, especially in the derivation of Pontryagin type maximum principle for stochastic optimal controls with recursive utilities (see \cite{hu2017stochastic,hu2018global,li2014optimal}). Considering its importance, Yong \cite{yong2020p} proposed a question whether a $L^{2}$-solution is an adapted $L^{p}$-solution for some $p>2$. In this section, we give a positive answer. We first obtain $L^{p}$-solution of \eqref{FBSDE5}: adapted solution $(X,Y,Z)$ such that:
\begin{equation*}
   E\left[\sup _{t \in[0, T]}|X_{t}|^{p}+\sup _{t \in[0, T]}|Y_{t}|^{p}+\left(\int_{0}^{T}|Z_{t}|^{2} d t\right)^{p / 2}\right] < \infty,
\end{equation*}
by adding the following integrability conditions assumptions:\\
\textbf{(B1)}(ii$^{\prime}$)
\begin{equation*}
\mathbb{E}\left[\left(\int_{0}^{T}|b(t,0,0)|ds\right)^{p}+\left(\int_{0}^{T}|f(t,0,0,0)|ds\right)^{p}+\left(\int_{0}^{T}|\sigma(t,0,0)|^{2}ds\right)^{\frac{p}{2}}+|h(0)|^{p}\right]<\infty,
\end{equation*}
for some $p>2$. 
\begin{corollary}\label{Lp corollary}
Assume assumptions (B1)(i)(ii$^{\prime}$) and (M1),(M3) hold, then for any $t\in[0,T]$ and $x\in\mathbb{R}$, the FBSDE \eqref{small lip FBSDE} has a unique $L^{p}$-solution on $[t,T]$. Further the random function $u$ defined on $\Omega\times[0,T]\times\mathbb{R}$ by 
\begin{equation*}
    u(\omega,t,x) = Y_{t}^{t,x}(\omega),
\end{equation*}
satisfies for any $t\in[0,T]$ and $s\in[t,T]$,
\begin{equation*}
    Y_{s}^{t,x} = u(s,X^{t,x}_{s}), ~~P\text{-}a.s.
    \end{equation*}
    and
    \begin{equation}
\left|\frac{u(t,x_{1})-u(t,x_{2})}{x_{1}-x_{2}}\right|\leq C \label{lp decoupling}
    \end{equation}
    for some constant $C$ only depending on $n,K,T$ and for all $x_1,x_2\in\mathbb{R}$ and $x_1\neq x_2$. In particular, FBSDE \eqref{FBSDE5} has a unique $L^p$-solution. 
\end{corollary}
\begin{proof}
 It is obvious that under assumption (B1)(i)(ii$^{\prime}$), all the conditions of the Theorem 2.3 in \cite{yong2020p} are satisfied, from which we get the unique $L^{p}$-solution in small duration. The remainder of the argument is analogous to that in Theorem \ref{Lipschitz corollary}.
\end{proof}

\indent Next we consider the $L^{p}$ estimates of FBSDEs, that is, for any $t\in[0,T]$ and $x\in\mathbb{R}$,
\begin{equation*}
   \mathbb{E}\left[\sup _{s \in[t, T]}|X_{s}^{t,x}|^{p}+\sup _{t \in[t, T]}|Y_{s}^{t,x}|^{p}+\left(\int_{t}^{T}|Z_{s}^{t,x}|^{2} d t\right)^{p / 2}\right] < C(1+|x|^{p})
\end{equation*}
for some $p\geq 2$, and now we need to strengthen the integrability condition to the following linear growth condition. \begin{itemize}
\item [\textbf{(B1)}(ii$^{\prime\prime}$)]
There exists constant $K>0$ such that for any $t \in[0, T]$, and $(x, y, z) \in \mathbb{R} \times \mathbb{R}^n\times \mathbb{R}^{n \times d}$,
\begin{equation*}
|b(t, x, y)|+|\sigma(t, x, y)|+|f(t, x, y, z)|+|h(x)| \leq K(1+|x|+|y|+|z|).
\end{equation*}
\end{itemize}
\begin{Theorem}
Let assumptions (B1)(i)(ii$^{\prime\prime}$) and (M1)(M3) hold, then for any $t\in[0,T]$ and $x\in\mathbb{R}$, FBSDE \eqref{small lip FBSDE} has a unique $L^{p}$-solution for any $p\geq 2$ and $L^{p}$ estimates hold, that is, \begin{equation}
\mathbb{E}\left[\sup _{t \leq s \leq T}\left|X_s^{t, x}\right|^p+\sup _{t\leq s \leq T}\left|Y_s^{t, x}\right|^p+\left(\int_s^T\left|Z_s^{t, x}\right|^2 d s\right)^{\frac{p}{2}}\right] \leq C\left(1+|x|^p\right).\label{global lp} \end{equation}\label{lp solution}
\end{Theorem}
\begin{proof}
First of all, for any $t\in[0,T]$, the existence and uniqueness of $L^{p}$-solution of FBSDE \eqref{small lip FBSDE} on $[t,T]$ is obtained from Corollary \ref{Lp corollary}, in particular, \eqref{lp decoupling} shows that the coefficients of FBSDE \eqref{small lip FBSDE} satisfy the same assumptions in the interval $[t,T]$. 
Applying  $L^{p}$ estimates results in small duration (see \cite{li2014optimal,yong2020p,delarue2002existence}), we can divide the time $[t,T]$ into $m$ intervals such that $t_{i+1}-t_{i} \leq \delta$, which depends on $K,C,p$, and there exists a constant $C^{(1)}$ such that 
\begin{equation}
\mathbb{E}\left\{\sup_{t_i \leq t \leq t_{i+1}}\left[\left|X_t^{t,x}\right|^p+\left|Y_t^{t,x}\right|^p\right]+\left(\int_{t_i}^{t_{i+1}} \left|Z_t^{t,x}\right|^2 d t\right)^{\frac{p}{2}}\right\} 
\leq \mathbb{E}\left[C^{(1)} (1+\left|X_{t_i}^{t,x}\right|^p)\right], \quad i=0,1,2,\cdots,m-1.
 \end{equation}
 We first consider the cases where $i=1,2$,
\begin{equation*}
\mathbb{E}\left[\sup _{t \leq s \leq t_{1}}\left|X_s^{t, x}\right|^p+\sup _{t \leq s \leq t_{1}}\left|Y_s^{t, x}\right|^p+\left(\int_t^{t_{1}}\left|Z_s^{t, x}\right|^2 d s\right)^{\frac{p}{2}}\right] \leq C^{(1)}\left(1+\left|x\right|^p\right),
\end{equation*}
and
\begin{equation*}
\mathbb{E}\left[\sup _{t_{1}\leq s \leq t_{2}}\left|X_s^{t, x}\right|^p +\sup _{t_{1} \leq s \leq t_{2}}\left|Y_s^{t, x}\right|^p+\left(\int_{t_{1}}^{t_{2}}\left|Z_s^{t, x}\right|^2 d s\right)^{\frac{p}{2}} \right] \leq \mathbb{E}\left[C^{(1)}\left(1+\left|X_{t_{1}}^{t, x}\right|^p\right)\right].
\end{equation*}
From the case $i=1$, we have
\begin{equation*}
\begin{aligned}
\mathbb{E}\left[C^{(1)}\left(1+\left|X_{t_{1}}^{t, x}\right|^p\right)\right] & \leq C^{(1)}\left(1+C^{(1)}\left(1+|x|^p\right)\right) \\
& \leq\left(C^{(1)}+\left(C^{(1)}\right)^2\right)\left(1+|x|^p\right).
\end{aligned}
\end{equation*}
Let $C^{(2)}=2 C^{(1)}+\left(C^{(1)}\right)^2$, it follows that
\begin{equation*}
C^{(1)}\left(1+|x|^p\right)+\mathbb{E}\left[C_1^{(1)}\left(1+\left|X_{t_{1}}^{t, x}\right|^p\right)\right] \leq C^{(2)}\left(1+|x|^p\right).
\end{equation*}
Adding on both sides of cases $i=1$ and $i=2$, we have
\begin{equation*}
\mathbb{E}\left[\sup _{t\leq s \leq t_{2}}\left|X_s^{t, x}\right|^p +\sup _{t \leq s \leq t_{2}}\left|Y_s^{t, x}\right|^p+\left(\int_t^{t_{1}}\left|Z_s^{t, x}\right|^2 d s\right)^{\frac{p}{2}}  +\left(\int_{t_{1}}^{t_{2}}\left|Z_s^{t,x}\right|^2 d s\right)^{\frac{p}{2}}\right] \leq C^{(2)}\left(1+|x|^p\right).
\end{equation*}
 Let $\hat{C}^{(2)}=2^{\frac{p}{2}} C^{(2)}$, we obtain
\begin{equation*}
\mathbb{E}\left[\sup _{t \leq s \leq t_{2}}\left|X_s^{t, x}\right|^p+\sup _{t\leq s \leq t_{2}}\left|Y_s^{t, x}\right|^p+\left(\int_t^{t_{2}}\left|Z_s^{t, x}\right|^2 d s\right)^{\frac{p}{2}}\right] \leq \hat{C}^{(2)}\left(1+|x|^p\right),
\end{equation*}
which is from the inequality $(a+b)^{k}\leq 2^{k}(a^{k}+b^{k})$, for $a,b\geq 0, k\geq 1$.

Then we can get the similar estimates for the case $i=3,4,\cdots,m-1$, and we obtain the $L^{p}$ estimates
\begin{equation*}
\mathbb{E}\left[\sup _{t \leq s \leq T}\left|X_s^{t, x}\right|^p+\sup _{t\leq s \leq T}\left|Y_s^{t, x}\right|^p+\left(\int_t^T\left|Z_s^{t, x}\right|^2 d s \right)^{\frac{p}{2}} \right]\leq \hat{C}^{(n)}\left(1+\left|x\right|^p\right),
\end{equation*}
now let $C= \hat{C}^{(m)}$, \eqref{global lp} holds, which ends the proof.
\end{proof}

We conclude this subsection by the following statements. The monotonicity conditions we introduced are a class of proper conditions which guarantee the  equivalence of $L^2$-solution and $L^p$-solution $(p>2)$ under linear growth condition. This provides a positive answer to a question proposed in \cite{yong2020p}. Moreover, $L^{p}$ estimates of FBSDEs are usually established only for small time horizon (see \cite{li2014optimal}) and it is quite challenging to get $L^{p}$ estimates globally. Our approach provides $L^{p}$ estimates globally under suitable conditions.

\section{Link to time-delayed BSDEs}
\paragraph{}
Since we do not require the non-degeneracy of the forward diffusion process $\sigma$ in our solvability result, we are able to build a natural connection between FBSDEs and time-delayed BSDEs, which allows us to get global solvability for some kind of time-delayed BSDEs. To the best of our knowledge, few researches concern the global solvability for time-delayed BSDEs. \cite{delong2010backward} obtained the existence and uniqueness of solution for time-delayed BSDEs and arbitrarily large time horizon, however, only for special kind of generators. When the generator has quadratic growth and depends only on recent delay of value process, \cite{briand2013simple} provided a constructive approach to get the existence and uniqueness of solution, see also Luo \cite{luo2020type} for multidimensional case.\\
\indent Relying on our global solvability result for diagonal FBSDEs, we get global existence and uniqueness of  solutions for BSDE with a special kind of delay, which extends the results of \cite{luo2018bsdes}. Indeed, we consider the following BSDE with delay in the value process:
\begin{equation}
Y_{t}=\xi+\int_{t}^{T} g\left(s, \int_{0}^{s} Y_{r} d r, Z_{s}\right) d s-\int_{t}^{T} Z_{s} d W_{s}, \quad t \in[0, T].    \label{BSDE}
\end{equation}
\indent We denote by $\mathcal{D}^{1,2}$ the space of all Malliavin differentiable random variables and for $\xi \in \mathcal{D}^{1,2}$ denote by $D_{t} \xi$ its Malliavin derivative. We refer to \cite{nualart2006malliavin} for a thorough treatment of the theory of Malliavin calculus. Let $K$ be a positive constant, we make the following assumptions:
\vspace{0.5em}
\begin{itemize}
    \item [\textbf{(C0)}]$g:\Omega \times[0, T] \times \mathbb{R}\times \mathbb{R}^{d} \rightarrow \mathbb{R}$ is a continuous function and $g(t,y,z) $ is decreasing with $y$.
    \vspace{0.5em}
    \item [\textbf{(C1)}] g is deterministic and  there exists a non-decreasing function $\rho: \mathbb{R}_{+} \rightarrow \mathbb{R}_{+}$ such that
\begin{equation*}
\begin{aligned}
&\left|g(t, y, z)-g\left(t, y^{\prime}, z^{\prime}\right)\right| \leq K\left|y-y^{\prime}\right|+\rho\left(|z| \vee\left|z^{\prime}\right|\right)\left|z-z^{\prime}\right|,\\
&\left|g(t, y, z)-g\left(t, y^{\prime}, z\right)-g\left(t, y, z^{\prime}\right)+g\left(t, y^{\prime}, z^{\prime}\right)\right| \leq K\left|y-y^{\prime}\right|\left(\left|z-z^{\prime}\right|\right) 
\end{aligned}
\end{equation*}
for all $ t \in[0, T], y, y^{\prime} \in \mathbb{R}$  and $ z, z^{\prime} \in \mathbb{R}^{d}$.
\vspace{0.5em}
    \item [\textbf{(C2)}]$\xi$ is $\mathcal{F}_{T}$-measurable such that $\xi \in \mathcal{D}^{1,2}\left(\mathbb{R}\right)$ and
\begin{equation*}
\left|D_{t} \xi \right| \leq K , \quad t \in[0, T].
\end{equation*}
    \item [\textbf{(C3)}] It holds that
\begin{equation*}
\begin{aligned}
&\left|g(t, y, z)-g\left(t, y^{\prime}, z^{\prime}\right)\right| \leq K\left|y-y^{\prime}\right|+K\left(1+|z|+\left|z^{\prime}\right|\right)\left|z-z^{\prime}\right|,\\
&|g(t, y, z)| \leq K\left(1+|z|^{2}\right)
\end{aligned}
\end{equation*}
 for all $t \in[0, T], y, y^{\prime} \in \mathbb{R}$ and $z, z^{\prime} \in \mathbb{R}^{d}$.
 \vspace{0.5em}
    \item [\textbf{(C4)}] $\xi$ is $\mathcal{F}_{T}$-measurable such that there exists a constant $K \geq 0$ such that $|\xi| \leq K$.
    \vspace{0.5em}
    \item [\textbf{(C5)}]$g:[0, T] \times\Omega \times \mathbb{R} \times \mathbb{R}^{d} \rightarrow \mathbb{R}$ is a measurable function and $E\left[\int_{0}^{T}\left|g(t,0,0)\right|^{2} d t\right]<+\infty$. Moreover, the function $g(t,\cdot, \cdot)$ is continuous for each $t \in[0, T]$ such that
\begin{equation*}
\left|g(t,y, z)-g\left(t,y^{\prime}, z^{\prime}\right)\right| \leq K(|y-y^{\prime}|+|z-z^{\prime}|)
\end{equation*}
for all $ y,y^{\prime} \in \mathbb{R}$ and $z,z^{\prime} \in \mathbb{R}^{d}$.
\vspace{0.5em}
    \item [\textbf{(C6)}]$\xi$ is $\mathcal{F}_{T}$-measurable such that  $\xi \in L^{2}(\mathcal{F}_{T};\mathbb{R})$.
\end{itemize}

\begin{Theorem}
(i) If (C0)-(C2) hold, then BSDE \eqref{BSDE} admits a unique solution $(Y,Z) \in \mathcal{S}^{2}\left(\mathbb{R}\right) \times \mathcal{H}^{2}\left(\mathbb{R}^{d}\right)$  satisfying that Z is bounded.\\
(ii) If (C0), (C3)-(C4) hold, then there exist constants $C_{1}, C_{2} \geq 0$ such that BSDE \eqref{BSDE} admits a unique solution $(Y, Z) \in$ $\mathcal{S}^{2}\left(\mathbb{R}\right) \times \mathcal{H}^{2}\left(\mathbb{R}^{d}\right)$ satisfying $|Y| \leq C_{1}$ and $\|Z \cdot W\|_{B M O} \leq C_{2}$.\\
(iii) If (C0), (C5)-(C6) hold, then BSDE \eqref{BSDE} admits a unique solution $(Y,Z)\in \mathcal{S}^{2}\left(\mathbb{R}\right) \times \mathcal{H}^{2}\left(\mathbb{R}^{d}\right) $.
\end{Theorem}
\begin{proof}
Define the function $b: \mathbb{R} \rightarrow \mathbb{R} $ by setting for $y \in \mathbb{R},  b(y)=y $. For $t \in[0, T]$, put
\begin{equation*}
X_{t}=\int_{0}^{t} b\left(Y_{s}\right) d s.
\end{equation*}
Thus BSDE \eqref{BSDE} is equivalent to the following FBSDE
\begin{equation}
\left\{\begin{array}{l}
X_{t}=\int_{0}^{t} b\left(Y_{s}\right) d s, \\
Y_{t}=\xi+\int_{t}^{T} g\left(s, X_{s}, Z_{s}\right) d s-\int_{t}^{T} Z_{s} d W_{s}.
\end{array}\right.
\end{equation}
It is obvious that function $b,g$ satisfies assumption (M2). Therefore, the statements (i) (ii) (iii) follows immediately from Theorem \ref{Theorem:global}.
\end{proof}
\section{Applications}
\subsection{Application to stochastic differential games}
\indent We consider a game with $n$ players and a common state controlled by all players. The dynamics of controlled state is given by  
\begin{equation}
    d X_{t} = b(t,X_{t},\alpha_{t})dt + \sigma(t,X_{t})dW_{t},\quad X_{0} = x, \label{state}
\end{equation}
where $\alpha_{t}=( \alpha^{1},\alpha^{2},\cdots,\alpha^{n})$ and $\alpha^{i}$ is the control chosen by player $i$. We denote
\begin{equation*}
    \mathcal{A}_i :=\left\{\alpha^{i}: [0,T] \times \Omega \rightarrow \mathbb{R} \text{ is a progressive process such that }
      E\left[ \int_{0}^{T}|\alpha_{t}^{i}|^{2}dt \right] < \infty \right \},
\end{equation*}
which is the set of admissible controls for the player $i$,  and $\mathcal{A}=\prod_{i=1}^n \mathcal{A}_i $. The goal of player $i$ is to minimize the payoff functional
\begin{equation*}
    J^{i}(\alpha) = \mathbb{E}\left[\int_{0}^{T} f^{i}(s,X_{s},\alpha_{s}^{i})ds+g^{i}(X_{T})\right].
\end{equation*}
More precisely, the state dynamics and the running and terminal cost functions are specified by the following assumptions.\\
\textbf{Assumption (G).}
Let $b:\Omega \times [0,T] \times \mathbb{R} \times \mathbb{R}^{n} \rightarrow \mathbb{R}$, $\sigma:\Omega \times [0,T]\times \mathbb{R}\rightarrow \mathbb{R}^{d}$, $f^{i}:\Omega \times [0,T]\times \mathbb{R}\times \mathbb{R}\rightarrow \mathbb{R}$ are progressive measurable functions, $g^{i}: \Omega \times \mathbb{R} \rightarrow \mathbb{R}$ is $\mathcal{F}_{T}$-measurable and for $1\leq i \leq n$ and they satisfying the following properties:
    \begin{itemize}
        \item[(i)] $b$ is affine in $(x,\Vec{a})$, i.e, it is of the form \begin{equation*}
    b(t,x,\Vec{a}) =b_{0}(t)+b_{1}x+b_{2}\Vec{a}, \end{equation*}
    where the mapping $b_{0}(t),b_{1}(t):\Omega \times [0,T] \rightarrow \mathbb{R}$, $b_{2}(t):\Omega \times [0,T] \rightarrow \mathbb{R}^{n}$ are progressive measurable and bounded.
        \item [(ii)] 
        For all $(\omega,t)\in \Omega \times [0,T]$, the mapping $(x,a)\rightarrow f^{i}(t,x,a)$ is convex with $f^{i}$ being strict convex in $a$. 
          \item[(iii)] $f^{i}(t,\cdot,\cdot)$ and $ g(\cdot)$ are twice continuously differentiable. The partial derivatives $\partial_x f^{i}$ and $\partial_{a} f^{i}$ (respectively $\partial_x g^{i}$ ) are at most of linear growth in $(x, a)$ (respectively in $x)$, uniformly in $t \in[0, T]$. And the second order derivatives $f^{i}_{aa},f^{i}_{ax},f^{i}_{xx},g^{i}_{xx}$ are bounded.
          \item[(iv)]  There exist some $\epsilon_{0} > 0 $ such that $f^{i}_{aa} \geq \epsilon_{0}$ and $g^{i}_{xx}\geq 0$.
          \item[(v)] The function $\sigma (t,x)$ is uniformly  Lipschitz continuous with $x$ and at most linear growth in $x$.  
\end{itemize}
\indent Now, we are aiming to find the Nash equilibrium of this game.
\begin{definition}
A set of admissible strategy profiles $\hat{\alpha}=\left(\hat{\alpha}^{1}, \cdots, \hat{\alpha}^{n}\right) \in \mathcal{A}$ is said to be a Nash equilibrium for the game if:
\begin{equation*}
\forall i \in\{1, \cdots, n\}, \quad \forall \alpha^{i} \in \mathcal{A}_{i}, \quad J^{i}(\hat{\alpha}) \leqslant J^{i}\left(\alpha^{i}, \hat{\alpha}^{-i}\right),
\end{equation*}
where $\left(\alpha^{i}, \hat{\alpha}^{-i}\right)$ stands for the strategy profile $\left(\hat{\alpha}^{1}, \cdots, \hat{\alpha}^{i-1}, \alpha^{i}, \hat{\alpha}^{i+1}\right)$, in which the player $i$ chooses the strategy $\alpha^{i}$ while the others, indexed by $j \in\{1, \cdots, n\} \backslash\{i\}$, keep the original ones $\hat{\alpha}^{j}$.
\end{definition}
\indent We will characterize the Nash equilibrium by an appropriate FBSDE relying on stochastic maximum principle (see \cite{carmona2018probabilistic}). We define for each $i$ the (reduced) Hamiltonian $H^{i}:\Omega \times [0,T]\times \mathbb{R} \times \mathbb{R}\times \mathbb{R}^{n} \rightarrow \mathbb{R}$
by
\begin{equation*}
    H^{i}(t,x,y^{i},\Vec{a}) = (b_{0}(t)+b_{1}(t)x+b_{2}(t)\Vec{a})y^{i}+f^{i}(t,x,a^{i}),
\end{equation*}
where $\Vec{a}=(a^1,\cdots,a^n)$.
Then the corresponding optimal control process is given by 
\begin{equation*}
    b_{2}^{i}(t)y^{i}+f_{a}^{i}(t,x,\hat{\alpha}^{i}) = 0.
\end{equation*}
Under assumption (G)(iv), we can use the inverse function theorem to derive that there exists a uniform Lipschitz continuous function $h^{i}(t,x,\cdot):\mathbb{R}\rightarrow \mathbb{R}$, which is the inverse of the function $f_{a}^{i}(t,x,\cdot)$ such that 
\begin{equation*}
\hat{\alpha}^{i}(t,x,y^{i})
   = h^{i}(t,x,-b_{2}^{i}(t)y^{i}).
  \end{equation*}
\indent We consider the following FBSDE
\begin{equation}
\left\{\begin{aligned}
X_{t}&= x+\int_{0}^{t}(b_{0}(s)+b_{1}(s)X_{s}+b_{2}(s)\hat{\alpha}(s,X_{s},Y_{s}))ds + \int_{0}^{t} \sigma(s,X_{s})dW_{s},\\
Y_{t}^{i}&=g_{x}^{i}(X_{T})+\int_{t}^{T}b_{1}(s)Y_{s}^{i}+f^{i}_{x}(s,X_{s},\hat{\alpha}^{i}(s,X_{s},Y_{s}^{i}))d s+\int_{t}^{T}Z_{s}^{i}d W_{s}, \quad i=1,2,\cdots,n,
\end{aligned}\right.\label{game}
\end{equation}
where $\hat{\alpha}(s,x,y)=\left(\hat{\alpha}^1(s,x,y^1),\cdots,\hat{\alpha}^n(s,x,y^n)\right)$ for $s\in[0,T]$, $x\in\mathbb{R}$ and $y\in\mathbb{R}^n$.
\begin{Theorem}
Let assumption (G) holds, then the FBSDE \eqref{game} has a unique solution. Consequently, $\hat{\alpha}= \hat{\alpha}(t,X_t,Y_t)$ is a Nash equilibrium.
\end{Theorem}
\begin{proof}
By applying the chain rule to the equation
\begin{equation*}
    f_{a}^{i}(t,x,\hat{\alpha}^{i}(t,x,y^{i}))=-b_{2}^{i}(t)y^{i},
\end{equation*}
we obtain that
\begin{equation*}
\begin{aligned}
 &f_{ax}^{i}(t,x,\hat{\alpha}^{i}(t,x,y^{i}))+f_{aa}^{i}(t,x,\hat{\alpha}^{i}(t,x,y^{i}))\partial_{x}\hat{\alpha}^{i}(t,x,y^{i})=0,\\
 &f_{aa}^{i}(t,x,\hat{\alpha}^{i}(t,x,y^{i}))\partial_{y^{i}}\hat{\alpha}^{i}(t,x,y^{i}) = -b_{2}^{i}(t).
\end{aligned}
   \end{equation*}
  Then we get
  \begin{equation*}
  \begin{aligned}
   &\partial_{x}\hat{\alpha}^{i}(t,x,y^{i}) = -\frac{f_{ax}^{i}(t,x,\hat{\alpha}^{i}(t,x,y^{i}))}{f_{aa}^{i}(t,x,\hat{\alpha}^{i}(t,x,y^{i}))},\\
&\partial_{y^{i}}\hat{\alpha}^{i}(t,x,y^{i}) =-\frac{b_{2}(t)}{f_{aa}^{i}(t,x,\hat{\alpha}^{i}(t,x,y^{i}))}.
   \end{aligned} \label{h x derivative}
      \end{equation*}
 From the boundedness of $f_{ax}$, $f_{aa}$ and $b_{2}$, we obtain that $\hat{\alpha}(t,x,y)$ is Lipschitz continuous with respect to $(x,y)$. Local boundedness of $\hat{\alpha}(t,x,y)$ follows from the $\partial_{\alpha}f$ is linear growth in $(x,a)$.\\
 Observe that 
 \begin{equation}
 \begin{aligned}
  \partial_{x}(b_{1}(t)y_{t}^{i}+f_{x}^{i}(t,x,\hat{\alpha}^{i}(t,x,y^{i})))
  &= f_{xx}^{i}(t,x,\hat{\alpha}^{i}(t,x,y^{i}))+f_{xa}^{i} (t,x,\hat{\alpha}^{i}(t,x,y^{i}))\partial_{x}\hat{\alpha}^{i}(t,x,y^{i})\\
  & = f_{xx}^{i}(t,x,\hat{\alpha}^{i}(t,x,y^{i})) -\frac{f_{xa}^{i}f_{ax}^{i}}{f_{aa}^{i}}(t,x,\hat{\alpha}^{i}(t,x,y^{i}))\\
  &=\frac{f_{xx}^{i}f_{aa}^{i}-f_{xa}^{i}f_{ax}^{i}}{f_{aa}^{i}}(t,x,\hat{\alpha}^{i}(t,x,y^{i}))\geq 0,
 \end{aligned}\label{f1}
\end{equation}
 where the last inequality follows form the convexity of function $f$.\\
Moreover, 
\begin{equation}
\begin{aligned}
 \partial_{y^{i}}(b_{0}(t)+b_{1}(t)x+b_{2}(t)\hat{\alpha}(t,x,y))
 &= -b_{2}^{i}(t) \partial_{y^{i}}\hat{\alpha}^{i}(t,x,y^{i})\\
&=\frac{-(b_{2}^{i}(t))^{2}}{f_{aa}^{i}(t,x,\hat{\alpha}^{i}(t,x,y^{i}))}\leq 0,\quad i=1,2,\cdots,n.
\end{aligned}\label{b2}
\end{equation}
Combining \eqref{f1}, \eqref{b2} and $g_{xx} \geq 0$ from assumption (G)(iv), it is obvious that the conditions needed in Theorem \ref{Lipschitz corollary} are satisfied, then \eqref{control} admits a unique solution. Then it follows from the Theorem 2.16 in \cite{carmona2018probabilistic} that $\hat{\alpha}$ is the Nash equilibrium.
\end{proof}
\subsection{Application to LQ problems with random coefficients}
\indent In this subsection, we study a  linear-quadratic stochastic control problem with random coefficients,  where the controlled state process is given by
\begin{equation*}
    d X_{t} = (A_{t}X_{t}+B_{t}u_{t})dt + \sigma(t,X_{t})dW_{t}, \quad X_{0} = x,
\end{equation*}
where $x \in \mathbb{R}$. The goal is to minimize the following cost functional 
\begin{equation*}
    J(u) = \mathbb{E}\left[\int_{0}^{T} (C_{s}X_{s}+D_{s}u_{s}+\frac{E_{s}}{2}X_{s}^{2}+\frac{F_{s}}{2}u_{s}^{2})ds+g(X_{T})\right].
\end{equation*}
The admissible set on which the cost function $J$ is minimized is 

\begin{equation*}
    \mathcal{U} :=\left\{u:[0,T] \times \Omega \rightarrow \mathbb{R} \text{ is a progressive process such that }  E\left[ \int_{0}^{T}|u_{t}|^{2}dt \right] < \infty \right \}.
\end{equation*}
\indent Now we list the assumptions on the coefficients appearing in the state dynamics and in the objective functional.\\
\textbf{Assumption (LQ)}
\begin{itemize}
    \item[(i)] $A,B,C,D,E,F$ are real-valued bounded  stochastic processes, and $E$ is non-negative, $F$ is positive.
    \item[(ii)] $\sigma: \Omega \times [0,T] \times \mathbb{R} \rightarrow \mathbb{R}^{d}$ is progressive measurable, uniform Lipschitz continuous with $x$ and at most linear growth in $x$.
     \item[(iii)] $g: \Omega \times \mathbb{R} \rightarrow \mathbb{R}$ is $\mathcal{F}_{T}$-measurable and differentiable in $x$. Moreover, $g_{x}$ is increasing in $x$,  uniform Lipschitz continuous with $x$ and at most linear growth in $x$.
\end{itemize}
The (reduced) Hamiltonian $H:\Omega \times [0,T]\times \mathbb{R} \times \mathbb{R}\times \mathbb{R} \rightarrow \mathbb{R} $ is given by 
\begin{equation*}
    H(t,x,y,u) = (A_{t}x+B_{t}u)y+C_{t}x+D_{t}u+\frac{E_{t}}{2}x^{2}+\frac{F_{t}}{2}u^{2}.
\end{equation*}
The minimizer of the Hamiltonian is
\begin{equation*}
    \hat{u}(t,y) = \frac{-B_{t}y-D_{t}}{F_{t}}.
\end{equation*}
The adjoint FBSDE associated with the stochastic maximal principle
\begin{equation}
\left\{\begin{array}{l}
X_{t}=x+\int_{0}^{t}(A_{s}X_{s}+B_{s}(\frac{-B_{s}Y_{s}-D{s}}{F_{s}}))ds+ \int_{0}^{t}\sigma(s,X_{s})dW_{s}, \\
Y_{t}= g_{x}(X_{T})+\int_{t}^{T}(A_{s}Y_{s}+E_{s}X_{s}+C_{s})d s+\int_{t}^{T}Z_{s}dW_{s}.
\end{array}\right.\label{control}
\end{equation} 
\begin{Theorem}
Suppose assumption (LQ) holds, then FBSDE \eqref{control} admits a unique solution. Moreover, $\hat{u}$ defined by $\hat{u}_{t}= \frac{-B_{t}Y_{t}-D_{t}}{F_{t}}$ is a optimal control over the interval $[0, T]$.
\end{Theorem}
\begin{proof}
It can be easy verify that under assumption (LQ), the assumptions (B1)(i)(ii),(M1),(M3) hold, then it follows from Theorem \ref{Lipschitz corollary}, there exists a unique solution to FBSDE \eqref{control}. The optimal control statement is directly from the stochastic maximum principle.
\end{proof}
\bibliographystyle{siam}
\bibliography{mybib}
\end{document}